\documentclass[11pt, oneside, reqno]{amsart}
\usepackage{amsmath,amsthm,amssymb,mathtools,tocvsec2,pdflscape,cite}
\usepackage{latexsym}
\usepackage{dsfont}
\usepackage{color,soul}
\usepackage[utf8]{inputenc}
\usepackage[T1]{fontenc}
\usepackage[mathscr]{euscript}
\usepackage[tmargin=1in,bmargin=1in,rmargin=1in,lmargin=1in]{geometry}
\usepackage[hidelinks]{hyperref}
\usepackage[shortlabels]{enumitem}
\usepackage{tikz-cd} 
\usepackage{bbm}

\allowdisplaybreaks
\numberwithin{equation}{section}

\newtheorem{theorem}{Theorem}[section]
\newtheorem{lemma}[theorem]{Lemma}
\newtheorem{proposition}[theorem]{Proposition}

\theoremstyle{definition}

\newtheorem{example}[theorem]{Example}
\newtheorem*{question}{Question}

\newcommand{\beas}{\begin{eqnarray*}}
\newcommand{\eeas}{\end{eqnarray*}}
\newcommand{\bes} {\begin{equation*}}
\newcommand{\ees} {\end{equation*}}
\newcommand{\be} {\begin{equation}}
\newcommand{\ee} {\end{equation}}
\newcommand{\bea} {\begin{eqnarray}}
\newcommand{\eea} {\end{eqnarray}}
\newcommand{\beals} {\begin{align*}}
\newcommand{\eeals} {\end{align*}}
\newcommand{\beal} {\begin{align}}
\newcommand{\eeal} {\end{align}}

\newcommand{\hol}{\mathcal {O}}

\newcommand{\rt}{\mathcal{\widetilde{R}}}
\newcommand{\cont}{\mathcal C}

\newcommand{\C}{\mathbb C}
\newcommand{\Cn}{{\mathbb{C}^n}}
\newcommand{\N}{\mathbb N}
\newcommand{\Z}{\mathbb Z}
\newcommand{\R}{\mathbb{R}}

\newcommand{\Le}{\mathbb L}

\newcommand{\ltt}{\left(}
\newcommand{\rtt}{\right)}

\newcommand{\ptt}{\check{p}}

\newcommand{\Om}{\Omega}

\newcommand{\zbar}{\overline z}

\newcommand\smpartl[2]{\frac{\partial{#1}}{\partial{#2}}}
\newcommand\partl[2]{\dfrac{\partial{#1}}{\partial{#2}}}

\begin{document}

\begin{abstract}
   {The space of Laplace transforms of holomorphic Hardy-space functions have been characterized as weighted Bergman spaces of entire functions in two cases: that of planar convex domains (Lutsenko--Yumulmukhametov, 1991), and that of strongly convex domains in higher dimensions (Lindholm, 2002). In this paper, we establish such a Paley--Wiener result for a class of (weakly) convex Reinhardt domains in $\C^2$ that are well-modelled by the so-called egg domains. We consider Hardy spaces on these domains with respect to a canonical choice of boundary Monge--Amp{\`e}re measure. This class of domains was introduced by Barrett--Lanzani (2009) to study the $L^2$-boundedness of the Leray transform in the absence of either strongly convexity or $\cont^2$-regularity. The boundedness of the Leray transform plays a crucial role in understanding the image of the Laplace transform. As a supplementary result, we expand the known class of convex Reinhardt domains for which the Leray transform is $L^2$-bounded (with respect to the aforementioned choice of boundary measure).
    Finally, we also produce an example to show that the Lutsenko--Yumulmukhametov result cannot be expected to generalize to all convex domains in higher dimensions.}
\end{abstract}

\title{The Laplace and Leray transforms on some (weakly) convex domains in $\C^2$}
\author{Agniva Chatterjee}
\keywords{Laplace transform, Leray transform, Paley-Wiener theorems, Hardy spaces.}
\subjclass[2020]{32A26. 44A10. 30H10. 32A36.}
\address{Department of Mathematics, Indian Institute of Science, Bangalore 560012, India}
\email{agnivac@iisc.ac.in}
\maketitle

\section{Introduction}

Let $\mathcal X$ be a topological vector space of functions on a subset of $\Cn$ such that all exponential functions of the form $\zeta\mapsto e^{\langle z,\zeta\rangle}$, $z\in\Cn$, belong to $\mathcal X$. It is sometimes useful to identify $\mathcal X'$, the dual space of $\mathcal X$, with a space of functions on $\Cn$ via the Laplace transform
    \bes
        \mathcal L(f)(z)=f(e^{\langle z,\cdot\rangle}),\quad f\in\mathcal X'.
    \ees
Some classical examples of $\mathcal X'$, for which this identification has been used to great effect, include the space of $L^2$-functions with compact convex support in $\R^n$ (the Paley--Wiener theorem), the space of distributions with compact convex support in $\R^n$ (the Paley-Wiener-Schwartz theorem), the space of hyperfunctions with convex compact support in $\R^n$ (the Paley-Wiener-Martineau theorem), and the space of analytic functionals carried by compact convex sets in $\C^n$ (the P{\'o}lya-Martineau-Ehrenpreis theorem). More recently, results in this spirit have appeared in the context of interpolation and sampling, \cite{Ro07, Be97}, and in the proofs of certain geometric inequalities for convex bodies, \cite{MR22, Be22}.

\subsection{Laplace transforms of holomorphic Hardy functions} In this paper, we consider the case when $\mathcal X=\mathcal X'$ is a Hardy space of holomorphic functions on a convex Reinhardt domain of the type considered by Barrett--Lanzani \cite{BaLa09}. Our motivation is as follows. Given a bounded convex domain $\Om\subset\Cn$ and a Radon measure $\mu$ on $b\Om$, let $L^2(b\Om,\mu)$ denote the space of complex-valued functions on $b\Om$ that are square-integrable with respect to $\mu$, and $\mathcal A(\Om)$ denote the space of continuous $\C$-valued functions on $\overline\Om$ that are holomorphic on $\Om$. Consider the Laplace transform $\mathcal L$ on $L^2(b\Om,\mu)$ given by
    \bes
        \mathcal L:f\mapsto F(z)=\int_{b\Om}  \overline{f(\zeta)}\,e^{\langle z,\zeta\rangle}\,d\mu(\zeta),\quad  \forall f\in L^2(b\Om,\mu).
    \ees
Since $\mathcal L$ annihilates the orthogonal complement of
    \bes
        \mathcal H^2(\Om,\mu)=\text{the $L^2(\mu)$-closure of $\mathcal A(\Om)|_{b\Om}$ in }L^2(b\Om,\mu),
    \ees
we might as well consider the restriction of $\mathcal L$ to $\mathcal H^2(\Om,\mu)$ (also denoted by $\mathcal L$). 

The P{\'o}lya-Martineau-Ehrenpreis theorem gives a necessary condition for an entire function to be the Laplace transform of some $f\in\mathcal H^2(\Om,\mu)$. In some cases, it is possible to completely characterize the range of $\mathcal L$ as a weighted Bergman space of the form
        \bes
A^2(\Cn,\omega)=\left\{F\in\hol(\Cn):\Vert F\Vert_\omega^2=\int_\Cn|F|^2d\omega<\infty\right\}      \text{ with norm }\Vert\cdot\Vert_\omega.
    \ees
One advantage of shifting the analysis to a weighted Bergman space is the availability of H{\"o}rmander's $L^2$-theory to construct functions with desired properties. The following rather comprehensive result completely tackles the case of planar convex domains when $\mu$ is the arc-length measure; {see \cite{LuYl91}}.

\begin{theorem}[Lutsenko--Yumulmukhametov]\label{th:Lu and Yul}   Let $\Om\subset\C$ be a bounded convex domain, and $\sigma_\Om$ be the arc-length measure on $b\Om$. Let $H_\Om$ denote the support function of $\Om$ (see \eqref{eq:supp}). Then, $\mathcal L$ is a normed space isomorphism between $\mathcal H^2(\Om,\sigma_\Om)$ and $A^2(\C,\widetilde\omega_{\Om})$, where 
    \be\label{E:LY}
        \widetilde\omega_\Om(z)=\Vert e^{\langle z,\cdot\rangle}\Vert^{-2}_{L^2(\sigma_\Om)}\Delta H_\Om(z),\quad z\in\C,
        \ee
        and $\Delta H_\Om$ is the Riesz measure of the convex function $H_\Om$.
\end{theorem}
Special cases of this theorem were previously proved by Levin (polygons), Likht (discs), Katsnel'son and Lyubarskii (strongly convex domains); see \cite[Appendix I]{Le64}, \cite{Lik64}, \cite{Ka65} and \cite{Ly88}, respectively. The case of strongly convex domains in $\Cn$ is also well-understood. 

\begin{theorem}[Lindholm]\label{th:lindholm} Let $\Om\subset \Cn$ be a bounded strongly convex domain, and $\sigma_\Om$ be the surface area measure on $b\Om$. Let $H_\Om$ denote the support function of $\Om$. Then, $\mathcal L$ is a normed space isomorphism between $\mathcal H^2(\Om,\sigma_\Om)$ and  $A^2(\Cn,\widetilde\omega_\Om)$, where 
\be\label{E:lindholm}
\widetilde\omega_\Om(z)=\Vert e^{\langle z,\cdot\rangle}\Vert^{-2}_{L^2(\sigma_\Om)}\left(dd^c H_\Om\right)^n(z),\quad z\in\Cn,
\ee
and $(dd^c H_\Om)^n$ is the Monge--Amp{\`e}re measure of $H_\Om$.
\end{theorem}
The measure $\widetilde\omega_\Om$ in \eqref{E:lindholm} is different from the one appearing in Lindholm's original text \cite{Li02}, but we show in Proposition~\ref{pr:str cvx weight} below that $\widetilde\omega_\Om$ works equally well. Theorem~\ref{th:lindholm} raises the 
\begin{question}Are the Laplace transforms of holomorphic Hardy functions on weakly convex domains in $\Cn$ characterized by a weighted Bergman condition?
\end{question}

A natural class of model domains in this context are the so-called egg domains:    \be\label{E:eggs}
        \left\{(z_1,z_2)\in\C^2:a_1|z_1|^{p_1}+a_2|z_2|^{p_2}<1\right\},\quad a_1,a_2\in(0,\infty), p_1,p_2\in(1,\infty).
    \ee
These are also sometimes called complex ellipsoids. 
We show that, instead of the surface area measure $\sigma_\Om$, if one considers the {\em boundary Monge--Amp{\`e}re measure} $\mu_\Om$ associated to the Minkowski functional of $\Om$, see \eqref{eq:defn MA min}, then Theorem~\ref{th:lindholm} generalizes to a class of Reinhardt domains in $\C^2$ that are modelled by egg domains.
Note that the measures $\mu_\Om$ and $\sigma_\Om$ are equivalent for $\Om$ as in Theorems~\ref{th:Lu and Yul} and ~\ref{th:lindholm}, but not for the domains considered in this paper. For some remarks on this particular choice of measure, see Section~\ref{sub:bdy meas}.


\begin{theorem}\label{th:main} Let $\Om\subset \C^2$ be a bounded convex Reinhardt domain that is $\cont^2$-smooth and strongly convex away from $\mathcal Z=\left\{(z_1,z_2)\in\C^2:z_1z_2=0\right\}$. Given $\zeta=(\zeta_1,\zeta_2)\in b\Om\setminus\mathcal Z$, 
let $\kappa_1(\zeta)$, $\kappa_2(\zeta)$, and $\kappa_3(\zeta)$ be the principal curvatures of $b\Om$ at $\zeta$, where the first two correspond to the directions $\left(i\zeta_1,0\right)$ and $\left(0,i\zeta_2\right)$, respectively. Let 
    \be\label{E:main}
        \omega_\Om(z)=\Vert e^{\langle z,\cdot\rangle}\Vert^{-2}_{L^2(\mu_\Om)}\left(dd^c H_\Om\right)^2(z),\quad z\in\C^2.
    \ee
If the function $\kappa_\Om=(\kappa_1\kappa_2)^{-1}\kappa_3$ is uniformly continuous and bounded away from zero on $b\Om\setminus\mathcal Z$, then $\Om$ has $\cont^1$-smooth boundary, and $\mathcal L$ is a normed space isomorphism between $\mathcal H^2(\Om,\mu_\Om)$ and $A^2(\C^2,\omega_\Om)$.
\end{theorem}

See Section~\ref{SS:normconstants} for bounds on $\Vert\mathcal L\Vert_{op}$ and $\Vert\mathcal L^{-1}\Vert_{op}$ for $\Om$ as above. On a strictly convex Reinhardt domain, the function $\kappa_\Om=(\kappa_1\kappa_2)^{-1}\kappa_3$ is well-defined on $b\Om\setminus\mathcal Z$. The additional hypothesis on $\kappa_\Om$ in Theorem~\ref{th:main} forces the domain to be in the class $\rt$ defined in \eqref{eq:rtilde}.
All domains of finite type in $\rt$, including the egg domains, satisfy the hypothesis of Theorem~\ref{th:main}; see Proposition~\ref{pr:infinite type}. Although the conditions on $\Om$ may seem very restrictive, we show that even within the class of convex Reinhardt domains, the conclusion of Theorem~\ref{th:main} fails to hold if $\kappa_3$ is allowed to vanish; see Theorem~\ref{th:negative}.


\subsection {Boundedness of the Leray transform} A key ingredient in the proof of Theorem~\ref{th:main} is the $L^2$-boundedness of the so-called Leray transform. Since the Leray transform is of independent interest, we state the relevant result separately. Classically, the {\em Leray integral} of a function $f$ on the boundary of a bounded $\cont^2$-smooth convex domain $\Om\subset\Cn$ is defined as 
\be\label{eq:Leray transform}
\Le(f)(z)=\int_{b\Om}\frac{f(\zeta)}{\langle\partial \rho(\zeta),\zeta-z\rangle^n} \lambda_\rho(\zeta),\quad z\in\Om,\quad 
\ee
where $\rho$ is a $\cont^2$-smooth defining function of $\Om$, $\partial\rho$ is shorthand for $(\partial\rho/{\partial\zeta_1},...,\partial\rho/{\partial\zeta_n})$,  and  $\lambda_\rho$ is the {\em Leray--Levi measure} given by the pull-back of the form $ d^c\rho\wedge\ltt dd^c\rho \rtt^{n-1}$ under the inclusion of $b\Om$ into $\C^n$. Due to the convexity of $\Om$, $\Le$ is a well-defined operator from $L^2(b\Om,\lambda_\rho)$ to $\hol(\Om)$. It reproduces holomorphic functions that are continuous up to the boundary; see \cite[Theorem 3.14]{Ra98}. If there is a dense subspace $\mathcal D\subset L^2(b\Om,\lambda_\rho)$ such that holomorphic functions in $\Le(\mathcal D)$ admit boundary values in $L^2(b\Om,\lambda_\rho)$, then one obtains a densely-defined operator $\Le_b:\mathcal D \rightarrow L^2(b\Om,\lambda_\rho)$ given by
\beas
    \Le_b:f&\mapsto& \text{b.v. }\Le(f).
\eeas
This is the {\em Leray transform} on $\Om$, and one can ask the 

\begin{question}
Does $\Le_b$ extend to a bounded operator from $L^2(b\Om,\lambda_\rho)$ onto $\mathcal H^2(\Om,\lambda_\rho)$?    
\end{question}

If yes, then the reproducing property of $\Le$ allows for functions in $\mathcal H^2(\Om,\lambda_\rho)$ to be decomposed into rational functions. In the plane, $\mathbb L_b$ is the Cauchy transform, and $\lambda_\rho$ is equivalent to the arc-length measure $\sigma_\Om$ on $b\Om$. The $L^p(\sigma)$-regularity, $1<p<\infty$, of the Cauchy transform, was established in the general setting of Lipschitz domains in the works of Calder{\' o}n \cite{Ca77}, Coifman-McIntosh-Meyer \cite{CMM82}, and David \cite{Da87}. For $\cont^2$-smooth domains in higher dimensions, the $L^2(\lambda_\rho)$-regularity of $\Le_b$ was established for egg domains by Hansson \cite{Ha99}, and for strongly convex bounded domains by Kerzman--Stein \cite{KS78} and Lanzani--Stein \cite{LS14}. Lanzani--Stein \cite{LS17}, {\cite{LaSt17},} also gave a counterexample to $L^2(\lambda_\rho)$-regularity of $\Le_b$ when strong convexity is dropped. Some unbounded hypersurfaces were considered by Barrett--Edholm {\cite{BE20}}, \cite{BE21}, {Edholm--Shelah\cite{EdSh25}, and Ha--Trung\cite{HaTr24}}. 

The Leray integral can be defined for any $\cont^1$-smooth convex domain $\Om\subset\Cn$ by replacing $\rho$ and $\lambda_\rho$ in \eqref{eq:Leray transform} by the Minkowski functional $m_\Om$ of $\Om$, and the boundary Monge--Amp{\`e}re measure $\mu_\Om$ associated to $m_\Om$, respectively. When $\Om$ is $\cont^2$-smooth, this coincides with the original definition of $\Le$. For $\cont^{1,1}$ domains, the $L^p(\mu_\Om)$-boundedness of $\Le_b$, $1<p<\infty$, was established for strongly $\C$-linearly convex domains by Lanzani--Stein \cite{LS14}. The class 
    \bea\label{eq:rtilde}
    \widetilde{\mathcal R}=\{\Om\subset\C^2: \Om\   \text{is a bounded convex $\cont^1$-smooth Reinhardt domain}\quad&& \notag\\
   \qquad \text{that is $\cont^2$-smooth and strongly convex away from $\mathcal Z$}\}.&&
    \eea
was studied by Barrett--Lanzani as a generalization of egg domains. They showed that for each $\Om\in\widetilde{\mathcal R}$, $\Le_b$ is a densely-defined operator on $L^2(b\Om,\mu)$, where $\mu$ belongs to a class of rotation-invariant finite measures on $b\Om$, which always includes $\mu_\Om$. They further described a subclass of domains $\mathcal R\subset\widetilde{\mathcal R}$ for which $\Le_b$ extends as a bounded projection operator from $L^2(b\Om,\mu)$ onto $\mathcal H^2(\Om,\mu)$. The class $\mathcal R$ has a technical definition, but essentially consists of domains in $\widetilde{\mathcal R}$ that are osculated by weighted $L^p$-balls even at the axes. For $\mu=\mu_\Om$, we describe a larger class of domains in $\widetilde{\mathcal R}$ for which $\Le_b$ is $L^2(\mu_\Om)$-bounded. Membership in this class imposes a weaker condition on the limiting behavior of $\kappa_1, \kappa_2$ and $\kappa_3$ than in $\mathcal R$; see Example~\ref{ex:rprime}.  

\begin{theorem}\label{th:Leray} Let $\Om\in\rt$, and $\Le_b$ denote the (densely-defined) Leray transform on $\Om$.
\begin{enumerate}[$(i)$]
\item If $(\kappa_1\kappa_2)^{-1}\kappa_3$ is uniformly continuous and bounded away from zero on $b\Om\setminus\mathcal Z$, then $\Le_b$ extends to a bounded projection operator from $L^2(b\Om,\mu_\Om)$ onto $\mathcal H^2(\Om,\mu_\Om)$.

\item If $(\kappa_1\kappa_2)^{-1}\kappa_3$ is either not bounded above or not bounded away from zero, then $\Le_b$ does not extend as a bounded operator to $L^2(b\Om,\mu_\Om)$.
\end{enumerate}
\end{theorem}

The case where $(\kappa_1\kappa_2)^{-1}\kappa_3$ is bounded above and away from zero, but does not extend continuously to $b\Om$, appears to be quite subtle. We have briefly elaborated on this in Section~\ref{se:rmk}.

\subsection{Relating the Leray and the Laplace transform}  Lindholm's original version of Theorem~\ref{th:lindholm} states that on strongly convex domains, $\mathcal L(\mathcal H^2(\Om,\sigma))=A^2(\Cn,\nu_\Om)$, where 
\be\label{E:origlind}
\nu_\Om(z)=e^{-2H_\Om(z)}\|z\|^{n-\frac{1}{2}}\left(dd^c H_\Om\right)^n(z),\quad z\in\Cn.
\ee
We show that this choice of measure works for precisely those domains in $\rt$ whose Leray transform is $L^2$-bounded. In order to recover Theorem~\ref{th:main} from this result, we must compare $A^2(\Cn,\nu_\Om)$ and $A^2(\Cn,\omega_\Om)$. While these spaces are the same (with comparable norms) for strongly convex domains, see Proposition~\ref{pr:str cvx weight}, we are only able to provide a sufficient condition for this equality to hold for a domain in $\rt$.


\begin{theorem}\label{th:main thm 2}
Let $\Om\in\widetilde{\mathcal R}$, and $\Le_b$ be the Leray transform on $\Om$. Let $\kappa_1, \kappa_2$ and $\kappa_3$ be as in Theorem~\ref{th:main}. Then, the following hold. 

\begin{enumerate}[$(i)$]
 \item 
The Leray transform $\Le_b$ extends to a bounded operator on $L^2(b\Om,\mu_\Om)$ if and only if $\mathcal L$ is a normed space isomorphism between $\mathcal H^2(\Om,\mu_\Om)$ and $A^2(\C^2,\nu_\Om)$.
\item
If the function $(\kappa_1\kappa_2)^{-1}\kappa_3$ is bounded above and away from zero on $b\Om\setminus\mathcal Z $, then the identity map is a normed space isomorphism between $
A^2(\C^2,\omega_\Om)$ and $A^2(\C^2,\nu_\Om)$.
\end{enumerate}
 \end{theorem}

See Section~\ref{SS:normconstants} for bounds on $\Vert\mathcal L\Vert_{op}$ and $\Vert\mathcal L^{-1}\Vert_{op}$ in the case when $\Le_b$ is $L^2(\mu_\Om)$-bounded.

We elaborate briefly on the role that the Leray transform  plays in studying the range of $\mathcal L$. On strongly convex domains, the Laplace and Leray transforms are related via the following commutative diagram:
\bes \begin{tikzcd}
\mathcal H^2(\Om,\mu_\Om) \arrow{dr}{\mathcal L} \arrow{r}{\mathcal S} & \mathcal L^2(b\Om^*,\mu_{\Om^*})  \arrow{r}{\Le_b} 
& \mathcal H^2(\Om^*,\mu_{\Om^*})  \\%
& A^2(\Cn,\omega_\Om)=A^2(\Cn,\nu_\Om) 
\arrow{ur}{\mathcal B}&  
\end{tikzcd}
\ees
Here, $\Om^*$ 
denotes the dual complement of $\Om$, see \eqref{de:dual complement}. 
There is a natural bijection $s:b\Om\rightarrow b\Om^*$, which induces the map $\mathcal S$. The map $\mathcal F_b=\Le_b\circ \mathcal S$ is the restriction to $\mathcal H^2(\Om,\mu_\Om)$ of the well-known Fantappi{\' e} transform, which is an isomorphism between $\hol'(\overline\Om)$ and $\hol(\text{int \!}\Om^*)$. The map $\mathcal B$ is the restriction to $A^2(\Cn,\nu_\Om)$ of the so-called Borel transform, see \cite[Definition~14]{Li02}. In order to establish that $\mathcal L$ is an isomorphism, Lindholm shows that $\mathcal F$ is  an isomorphism between $\mathcal H^2(\Om,\mu_\Om)$ and $\mathcal H^2(\Om^*,\mu_{\Om^*})$, and $\mathcal B$ is an isomorphism between $A^2(\Cn,\nu_\Om)$ and $\mathcal H^2(\Om^*,\mu_{\Om^*})$. In our case, it can be shown that
\begin{enumerate}
    \item [(a)]$\mathcal B:A^2(\C^2,\nu_\Om)\rightarrow \mathcal H^2(\Om^*,\mu_{\Om^*})$ is a well-defined isomorphism for all $\Om\in\widetilde{\mathcal R}$, 
    \item [(b)] $\mathcal F:\mathcal{H}^2\ltt \Om,\mu_\Om\rtt\rightarrow \mathcal{H}^2\ltt \Om^*,\mu_{\Om^*}\rtt$ is a well-defined isomorphism if and only if the Leray transform $\Le_b$ is a bounded projection operator from $L^2\ltt b\Om,\mu_\Om\rtt$ onto $\mathcal H^2\ltt \Om,\mu_\Om\rtt$.
\end{enumerate}
However, the symmetries of the domains allow us to take a more direct approach via series expansions. 



\subsection{A negative result} The following example shows that the measure $\omega_\Om$ will not yield as comprehensive a result in higher dimensions as Theorem~\ref{th:Lu and Yul}.

\begin{theorem}\label{th:negative}
    Let $\Om=\{(z_1,z_2)\in\C^2:|z_1|+|z_2|<1\}$. Then, 
    \bes \mathcal{L}\ltt\mathcal{H}^2\ltt \Om, \mu_\Om\rtt\rtt\subsetneq A^2\ltt\C^2,\nu_\Om\rtt\subsetneq A^2\ltt\C^2,\omega_\Om\rtt.
    \ees
\end{theorem}
Here, unlike in the case of $\rt$, the measure $(dd^c H_\Om)^2$ is supported on a three-dimensional surface in $\C^2$. This is analogous to the case of planar polygons, where the Riesz measure $\triangle H_\Om$ is supported on a union of real lines. However, for planar polygons, we have that $ A^2\ltt\C^2,\nu_\Om\rtt\subsetneq A^2\ltt\C^2,\omega_\Om\rtt=\mathcal{L}\ltt\mathcal{H}^2\ltt \Om, \mu_\Om\rtt\rtt$. Currently, we do not know of any $\cont^2$-smooth domains in $\C^2$ for which  $\mathcal{L}\ltt\mathcal{H}^2\ltt \Om, \mu_\Om\rtt\rtt \neq A^2\ltt\C^2,\omega_\Om\rtt$.

\subsection{Remarks on the choice of boundary measure}\label{sub:bdy meas} 
The suitability of $\mu_\Om$ for studying $\mathcal L$ was already indicated by Lindholm \cite{Li02}. We collect some observations to support this. 

\begin{enumerate}
    \item [(a)] Since Leray--Levi measures appear explicitly in the Leray kernel, boundedness of the Leray operator has often been studied with respect to these measures. Hansson makes a strong case for this measure in \cite{Ha99}. However, Leray--Levi measures are not independent of the choice of defining functions. For convex domains, the Minkowski functional allows for a canonical choice of defining function. Furthermore, this choice behaves well under the duality operation: $s^*(\mu_{\Om^*})=\mu_\Om$, where $s:b\Om\rightarrow b\Om^*$ is the natural bijection mentioned above.
    \item [(b)] Up to a constant factor, $\mu_\Om$ coincides with the special measure $\mu_0$ of order $0$ considered in \cite[\S 8]{BaLa09}; see \eqref{eq:boundary MA wrt param}. The $L^2$-adjoint of $\Le_b$ with respect to this measure is somewhat special, as demonstrated in \cite[Prop. 49]{BaLa09}.
    \item [(c)] When tackling the Borel transform $\mathcal B$, Lindholm uses a coarea-type formula over the level sets of $\mu_\Om$; see Proposition~\ref{pr:cov} for a version for $\rt$. Such a formula is needed to transfer data from $b\Om$ to $\Cn$. It is in this formula that the measures $\omega_\Om$ and $\mu_\Om$ appear naturally as the analogues of $\omega_{std}$ and $\sigma_\Om$ in the classical coarea formula.

\end{enumerate} 

\subsection{Bounds on the operator norms}\label{SS:normconstants} In applications, the classical Paley--Wiener theorem is used in conjunction with the Plancherel theorem. In all the results above, $\mathcal L$ is only shown to be a normed space isomorphism (as opposed to an isometric isomorphism). By tracking the constants in our proofs, we obtain the following bounds.
\begin{enumerate}
    \item [(a)] If $\Om\in\rt$ and $\Le_b$ extends as a bounded operator on $L^2(b\Om,\mu_\Om)$, then 
    \be\label{eq:norm bound laplace}
      \frac{\sqrt{\pi}}{e^2}c_\Om^{\frac{3}{2}}\Vert f\Vert_{\mu_\Om}^2
        \leq \Vert\mathcal L f\Vert_{\nu_\Om}^2\leq 
        \frac{5^3\sqrt{e}}{2^{15/2}\pi}C_\Om^{\frac{3}{2}}\Vert \Le_b\Vert_{op}^2\Vert f\Vert_{\mu_\Om}^2,
    \ee
for all $f\in \mathcal H^2(\Om,\mu_\Om)$, where $c_\Om^{-1}=\sup_{\Vert z\Vert=1}H_\Om(z)$ and $C_\Om^{-1}=\inf_{\Vert z\Vert=1}H_\Om(z).$
\medskip 

\item [(b)] If $\nu_\Om$ is replaced by the equivalent measure
    \bes        \widetilde\nu_\Om=e^{-2H_\Om(z)}H_\Om(z)^{\frac{3}{2}}\left(dd^c H_\Om\right)^2(z),\quad z\in\Cn,
    \ees
above, then $c_\Om^{3/2}$ and $C_\Om^{3/2}$ drop out of the above inequalities.  

\item [(c)] If $\Om\in\rt$ satisfies the hypothesis of Theorem~\ref{th:main thm 2} $(ii)$, $p_\ell=\inf\{p(\zeta):\zeta\in b\Om\}$, and $p_g=\sup\{p(\zeta):\zeta\in b\Om\}$, where $p$ is defined as in Section \ref{sub:domains}. Then
    \be\label{}
      k_1(\alpha_\Om C_\Om^{-1})^{\frac{3}{2}}\Vert F\Vert_{\nu_\Om}^2
        \leq \Vert F\Vert_{\omega_\Om}^2\leq 
        k_2(\beta_\Om c_\Om^{-1})^{\frac{3}{2}}\Vert F\Vert_{\nu_\Om}^2,\quad
        F\in A^2(\C^2,\omega_\Om),
    \ee 
    where $\alpha_\Om=p_g^{-1}$ and $\beta_\Om={p_g}^{-1}{(p_g-1)}$ when $p_\ell\geq 2$, $\alpha_\Om={p_g^{-1}}{(p_\ell-1)}$ and $\beta_\Om={p_g}^{-1}{(p_l-1)^{-1}}{(p_g-1)}$ when $1<p_\ell\leq 2$, $c_\Om$ and $ C_\Om$ are as in \eqref{eq:norm bound laplace}, and
$k_1,k_2$  are two positive constants that are independent of $\Om$.
\end{enumerate}

\subsection{Organization of the paper}
 In Section \ref{sub:prelim}, we elaborate on the domains, measures and function spaces appearing in the main results. In Section \ref{sub:tech tools}, we collect some technical tools that are used extensively in the rest of the paper. Of particular note are the following two ingredients from \cite{BaLa09}: 
 \begin{itemize}
\item [(i)] a special parametrization of any domain in $\rt$; see \eqref{eq:param s}, and
\item [(ii)] a decomposition of the Leray transform on any domain in $\rt$ into bounded rank-one projection operators; see Lemma ~\ref{le:Leray bdd criteria rt}.
 \end{itemize}
The proof of Theorem~\ref{th:Leray} is presented in Section \ref{sub:proof PW part 2}. It is a variation of the proof of Theorem~45 in \cite{BaLa09}, wherein one performs an asymptotic analysis of the norms of the rank-one operators obtained in (ii). The main distinction in our case is the weaker regularity of the parametrizing functions appearing in \eqref{eq:r_1 in s} and \eqref{eq:r_2 in s}. The proof of Theorem~\ref{th:main thm 2} is presented in Section \ref{sub:proof of PW}. The first part of the proof heavily relies on the series representations obtained in Section~\ref{subsec:series expansion}. The second half of the proof involves a direct comparison of the two measures in question. The proofs of Theorem \ref{th:main} and Theorem \ref{th:negative} are placed in Section~\ref{sub:counter}. Finally, in Section \ref{sub:comp str cvx}, we establish the equality of the two weighted Bergman spaces $A^2\ltt\C^n,\omega_\Om\rtt$ and $A^2\ltt\C^n,\nu_\Om\rtt$   
for strongly convex domains, thus unifying Lindholm's higher-dimensional result with the planar one. 
\vspace{2pt}
\newline
\textbf{Acknowledgements.} 
I am grateful to my thesis advisor, Purvi Gupta, for suggesting this problem and sharing many invaluable insights, as well as engaging in fruitful discussions throughout the course of this project. I would also like to thank her for helping me with the writing of this paper. I am also thankful to Koushik Ramachandran for directing me towards literature on regularly varying functions. Finally, I am grateful to the anonymous referees for their useful comments. 
\textbf{Funding.}This work is supported by a scholarship from the Indian Institute of Science, and the DST-FIST
programme (grant no. DST FIST-2021 [TPN-700661]).

\section{Preliminaries}\label{sub:prelim}
We collect here some preliminary results and observations regarding our main objects of study.  
\subsection{Notation} The following notation will be used throughout the paper. 
\begin{itemize}
    \item [(1)] $\mathbb B^n(R)$ denotes the Euclidean ball in $\C^n$ centered at the origin and of radius $R>0$.
    \item [(2)] $\mathbb B^n=\mathbb B^n(1)$ denotes the unit Euclidean ball in $\C^n$.
    \item [(3)] $\mathcal Z$ denotes the set $\{(z_1,z_2)\in\C^2:z_1z_2=0\}$.
    \item [(4)]
    $\langle\zeta,z\rangle$ denotes the pairing $\zeta_1z_1+\zeta_2z_2\cdots+\zeta_nz_n$, $\zeta=(\zeta_1,\cdots,\zeta_n),z=(z_1,\cdots,z_n)\in\C^n$.
    \item [(5)] $\omega_{std}$ denotes the Lebesgue volume measure on $\C^n$.
    \item [(6)] Given a bounded convex domain $\Om\subset\Cn$,
    \begin{itemize}
    \item $b\Om_+=b\Om\setminus \mathcal Z$,
    \item $m_\Om$ denotes the Minkowski functional of $\Om$,
    \item $H_\Om$ denotes the support function of $\Om$,
    \item $\mu_\Om$ denotes the boundary Monge-Amp{\`e}re measure of $\Om$,
    \item $\sigma_\Om$ denotes the Euclidean surface area measure on $b\Om$,
    \item $\omega_\Omega$ denotes the measure on $\Cn$ defined in \eqref{E:main},
    \item $\nu_\Omega$ denotes the measure on $\Cn$ defined in \eqref{E:origlind}.
    \end{itemize}
    
    \item [(7)] $\Vert.\Vert_{\nu}$ denotes the  $L^2(\nu)$-norm for a given measure $\nu$.
    \item [(8)] $\|.\|_{C(K)}$ denotes the supremum norm on the space of continuous functions on the compact set $K$.
    \item [(9)] $d=\partial+\overline{\partial}$ denotes the standard exterior derivative.
    \item [(10)] $d^c=\frac{i}{4\pi}\ltt{\overline{\partial}-\partial}\rtt.$
    \item [(11)] $\mathds 1_A$ denotes the indicator function of a set $A\subset\C^n$.
    \item [(12)] Given two $\R$-valued functions $f$ and $g$ on set a $X$, $f\approx g$ on $X$ denotes the existence of $C_1, C_2>0$ such that $C_1 g(x)\leq f(x)\leq C_2g(x)$ for all $x\in X$.
        \item [(13)] Given two positive measures $\mu$ and $\nu$ on a set $X$, 
        \begin{itemize}
            \item $\mu\ll\nu$ denotes that $\mu$ is absolutely continuous with respect to $\nu.$
            \item $\mu\approx\nu$ denotes that $\mu$ and $\nu$ are mutually absolutely continuous, and $d\mu/d\nu\approx 1$ as functions on $X$. 
        \end{itemize}
        \item [(14)] $\N$ denotes the set of all natural numbers, i.e., set of all positive integers union $\{0\}$.
\end{itemize}

\subsection{The class $\mathbf{\rt}$}\label{sub:domains}
Let $\Om$ be a bounded convex Reinhardt domain in $\C^2$ that is strongly convex away from $\mathcal Z$. Then $\Om$ is modelled by weighted $L^p$-balls in the following sense. Given  $\zeta=(\zeta_1,\zeta_2)\in b\Om_+=b\Omega\setminus\mathcal Z$, there exists a unique triplet $(p(\zeta),a_1(\zeta),a_2(\zeta))\in (1,\infty)\times(0,\infty)\times(0,\infty)$ such that the domain \bes \left\{(z_1,z_2)\in\C^2:a_1(\zeta)|z_1|^{p(\zeta)}+a_2(\zeta)|z_2|^{p(\zeta)}<1\right\}\ees
osculates $\Om$ at the point $\zeta$ up to order $2$; see \cite[Prop. 8]{BaLa09} whose proof does not depend on the regularity of $b\Om$ at points in $\mathcal Z$. The functions $a_1,a_2$ and $p$ are continuous on $b\Om_+$ and rotation-invariant in both the coordinates. The exponent function $p$ is expressible in terms of the principal curvatures of $b\Om_+$, as shown in Lemma~\ref{le:curvature and p} below. 

Due to its rotational symmetry, much of the geometry of $\Om$ is captured by its Reinhardt shadow. Let $\gamma_\Om=\left\{(r_1,r_2)\in \R^2_{\geq0}:(r_1,r_2)\in b\Om\right\}.$ Then, 
\be\label{eq:graph param gamma}
\gamma_\Om=\{(r_1,r_2):r_2=\phi(r_1),\ 0\leq r_1\leq b_1\}
\ee
for some $b_1>0$ and continuous function $\phi:[0,b_1]\rightarrow[0,\infty)$ satsifying
\begin{itemize}
    \item [$(i)$] $\phi>0$ on $[0,b_1)$ and $\phi(b_1)=0$,
    \item [$(ii)$] $\phi$ is $\cont^2$-smooth and strongly concave on $(0,b_1)$.
\end{itemize}
As a consequence of $(ii)$, one also obtains that
\begin{itemize}
    \item [$(iii')$] $\lim_{t\rightarrow 0^+}\phi'(t)\in (-\infty,0]$ and $\lim_{t\rightarrow b_1^-}\phi'(t)\in [-\infty,0)$.
\end{itemize}
Conversely, any $b_1>0$ and continuous $\phi:[0,b_1]\rightarrow [0,\infty)$ satisfying $(i)$ and $(ii)$ uniquely determine a bounded convex Reinhardt domain in $\C^2$ that is strongly convex off of $\mathcal Z$. In this paper, for a given $\Om,$ $b_1$ and $\phi$ denote the unique positive number and continuous function granted by \eqref{eq:graph param gamma}. The intercept of $\gamma_\Om$ with the $y$-axis, $\phi(0)$, is denoted by $b_2$.

\begin{lemma}\label{le:curvature and p}
Let $\Om$ be as above and $\rho$ be a   $\mathcal C^1$-smooth defining function of $\Om$ near $\zeta\in b\Om_+$.  Then, 
\bea
p(\zeta)&=&1+\ltt\frac{\kappa_3}{\kappa_1\kappa_2}\ltt\zeta\rtt
\times\frac{\Vert\nabla \rho(\zeta)\Vert}{2\operatorname{Re}\left<\partial{\rho}(\zeta),\zeta\right>}
\rtt \label{eq:reln between kappa an p}\\
&=&
1+\frac{\left|\zeta_1\right|\phi''(\left|\zeta_1\right|)\phi(\left|\zeta_1\right|)}{\phi'(\left|\zeta_1\right|)\ltt\phi(\left|\zeta_1\right|)-\left|\zeta_1\right|\phi'(\left|\zeta_1\right|)\rtt}. \label{eq:p in phi}
 \eea   
\end{lemma}
Equation \eqref{eq:reln between kappa an p} is proved in Section~\ref{sub:tech tools}, while \eqref{eq:p in phi} was already proved in \cite[Prop. 8]{BaLa09}. 

Recall that $\rt$ is the class of bounded $\cont^1$-smooth convex Reinhardt domains in $\C^2$ that are strongly convex away from $\mathcal Z$. We see in Section~\ref{subsec:param} that membership in $\rt$ imposes a condition on the limiting behavior of $p$ at the axes. In terms of the graphing function $\phi$ of its shadow, $\Om\in\rt$ if and only if, in addition to $(i)$ and $(ii)$, it also satisfies
\begin{itemize}
    \item[$(iii)$] $\lim_{t\rightarrow 0^+}\phi'(t)=0$ and $
    \lim_{t\rightarrow b_1^-}\phi'(t)=-\infty$.
\end{itemize}

It's clear from \eqref{eq:p in phi} that $(iii)$ does not guarantee the continuous extension of $p$ to $b\Om$. Let
\bea\label{D:newclass}
\mathcal R'&=&\left\{\Om\in\widetilde{\mathcal R}: p\ \text{extends as a continuous function from $b\Om$ to }(1,\infty)\right\}\\
&=&
\left\{\Om\in\widetilde{\mathcal R}: (\kappa_1\kappa_2)^{-1}\kappa_3\ \text{extends as a continuous function from $b\Om$ to }(0,\infty)\right\}.\notag
\eea
Several of the results in \cite{BaLa09} are proved for domains in the class $\mathcal{R}$; see \cite[Definition 16]{BaLa09}. It can be shown that domains in $\mathcal R$ are precisely those $\Om\in\rt$ for which $(p,a_1,a_2)$ extends continuously as a function from $b\Om
$ to $(1,\infty)\times(0,\infty)\times(0,\infty)$. It is clear the $\mathcal R\subseteq \mathcal R'$. See Example~\ref{ex:rprime} for a concrete example in $\mathcal R'\setminus\mathcal R.$

For domains in $\rt$, the limiting behavior of $p$ captures some information about the regularity and contact type of $\Om$ at $\mathcal Z$. Note the following result, for instance. Since it has no bearing on the proofs of our main results, the reader is directed to \cite{Ch25} for its proof.

\begin{proposition}\label{pr:infinite type} Let $\Om\in \rt$ and $w\in b\Om\cap \mathcal{Z}$. 
\begin{enumerate}[$(a)$]
    \item If $\lim_{\zeta\rightarrow w}p\ltt\zeta\rtt$ exists and is greater than $2$, then $b\Om$ is $\mathcal{C}^2$-smooth at $w$.
    \item If $b\Om$ is $\cont^2$-smooth at $w$, then  $\limsup_{\zeta\rightarrow w}p\ltt\zeta\rtt\geq 2$.
\end{enumerate}
 Suppose $b\Om$ is $\cont^\infty$-smooth at $w$. Let $m\in\Z_+$.

\begin{enumerate}
    \item [$(c)$] Then $\Om$ is of finite type $2m$ at $w$ if and only if  $\lim_{\zeta\rightarrow w}p\ltt\zeta\rtt$ exists and is equal to $2m$. 
    \item [$(d)$] Then $\Om$ is of infinite type at $w$ if and only if $\limsup_{\zeta\rightarrow w}p(\zeta)=\infty$.
\end{enumerate}
\end{proposition}






We note one final feature of the class $\rt$, which has a significant bearing on the proofs of the main results. Given a set $E\subset \C^n$, recall that its {\em dual complement} is the set given by
\be\label{de:dual complement}
E^*=\{z\in\C^n:\langle \zeta,z\rangle\neq 1,\,\forall\zeta\in E\}.
\ee
For a detailed discussion on the significance of the dual complement of a set, see \cite[Chapter 2]{APS}. We recall from \cite[\S 7]{BaLa09} that $\rt$ is closed under the action of taking dual complements, i.e., $\Om\in\rt$ if and only if (the interior of) $\Om^*\in\rt$. The class $\mathcal R'$ is also closed under the action of taking dual complements. 

\subsection{Monge--Amp{\`e}re measures associated to domains in $\rt$}\label{sub:Leray Levi} Recall that, for a bounded convex domain $\Om\subset\C^n$, the {\em support function of $\Om$} is the function $H_\Om:\Cn\rightarrow\mathbb R$ given by 
    \be\label{eq:supp}        H_\Om(z)=\sup_{\zeta\in\Om}\operatorname{Re}\langle\zeta,z\rangle,\qquad z\in\C^n.
    \ee
Since $H_\Om$ is a convex function, one can make sense of $(dd^c H_\Om)^n$ as a Radon measure on $\C^2$; see \cite[\S 2]{BeTa76}. In particular, if $n=2$, $\nabla H_\Om\in L^\infty_{loc}(\C^2)$, and 
\be\label{eq:def of MA}
\int_{\C^2} \psi (dd^c H_\Om)^2
=-\int_{\C^2} dH_\Om \wedge d^c H_\Om \wedge dd^c \psi,\qquad
\forall \psi\in \mathcal C_c^\infty(\C^2).
\ee
In general, $(dd^cH_\Om)^n$ may be mutually singular with respect to $\omega_{std}$, the Lebesgue measure on $\C^n$. However, when $\Om$ is strongly convex, $H_\Om\in\mathcal C^2(\C^n\setminus\{0\})$ (see \cite[\S 2.5]{Sc14}) and consequently $(dd^c H_\Om)^n\approx \Vert z\Vert ^{-n} \omega_{std}$. When $\Om\in\rt$, we have that $(dd^cH_\Om)^2\ll\omega_{std}$ on $\C^2$.


\begin{proposition}\label{pr:measure equality}
   Let $\Om\in \rt$. Then, $(dd^c H_\Om)^2= \mathfrak H_\Om\omega_{std}$ as measures, where
\bea\label{eq:Def of the MA fns}
\mathfrak H_\Om(z)=\begin{cases}
    \frac{2}{\pi^2}\left(\dfrac{\partial^2 H_{\Om}}{\partial z_1 \partial {\zbar_1}}\dfrac{\partial^2 H_{\Om}}{\partial z_2 \partial{\zbar_2}}-\dfrac{\partial^2 H_{\Om}}{\partial z_1 \partial \zbar_2}\dfrac{\partial^2 H_{\Om}}{\partial z_2 \partial{\zbar_1}}\right)(z),& \text{if }z\notin\mathcal Z,\\
    0,& \text{if }z\in\mathcal Z.
    \end{cases}
\eea
\end{proposition}
\begin{proof} Let $\mu=(dd^cH_\Om)^2\Big|_{\C^2\setminus\mathcal Z}$ and $\nu=\mathfrak H_\Om\omega_{std}\Big|_{\C^2\setminus\mathcal Z}$. By \cite[Corollary 3.5]{BeTa76}, $(dd^cH_\Om)^2$ does not charge $\mathcal Z$. Thus, it suffices to show that, $\mu=\nu$ on $\C^2\setminus \mathcal Z$. Since $\mu$ and $\nu$ are Radon measures on $\C^2\setminus \mathcal Z$, by the Riesz representation theorem for positive functionals, we must show that 
    \be\label{eq: test}
        \int_{\C^2\setminus \mathcal Z} \varphi d\mu
        =
        \int_{\C^2\setminus \mathcal Z} \varphi d\nu,\quad
        \forall \varphi\in C_c(\C^2\setminus\mathcal Z).
    \ee
Since $C_c^\infty(\C^2\setminus\mathcal Z)$ is dense in $C_c(\C^2\setminus\mathcal Z)$ in the sup-norm, and $\mu$ and $\nu$ act as continuous linear functionals on $C_c(\C^2\setminus\mathcal Z)$ via integration, it suffices to establish \eqref{eq: test} for $\varphi\in C_c^\infty(\C^2\setminus\mathcal Z)$. But this follows from \eqref{eq:def of MA} and Stokes' theorem as $H_\Om$ is $\cont^2$-smooth on $\C^2\setminus\mathcal Z$.
\end{proof}

Following Demailly \cite{De85}, one can also consider Monge--Amp{\`e}re boundary measures on the boundary of a bounded convex domain $\Om\subset\C^n$. Let $\rho:\overline\Om\rightarrow (-\infty,0]$ be a continuous convex exhaustion function of $\Om$, i.e., satisfying $\{z\in\overline\Om:\rho(z)<-c\}\Subset\Om$ for all $c>0$, such that $\int_{\Om} (dd^c\rho)^n<\infty$. Then, by \cite[Theorem 3.1]{De85}, $\{(dd^c\max\{\rho,r\})^2:r<0\}$ is a family of Radon measures on $\overline\Om$ that converge weakly as $r\rightarrow 0$, and
\be\label{E:ma}
    \mu_\rho=\lim_{r\rightarrow 0}
    (dd^c\max\{\rho,r\})^2
\ee
is a Radon measure supported on $b\Om$. In particular, let
    \be\label{eq:defn MA min}
        \mu_{_\Om}=\lim_{r\rightarrow 0}
    (dd^c\max\{m_\Om-1,r\})^2,
    \ee
where $m_\Om$ is the {\em Minkowski functional} of $\Om$ given by $m_\Om(z)=\inf\{t>0:\frac{z}{t}\in \Om\}$. When the exhaustion function $\rho$ is $\mathcal C^2$-smooth on a neighborhood of $w\in b\Om$ and $\nabla\rho(w)\neq 0$, then $\mu_\rho(w)$ coincides with the so-called {\em Leray--Levi measure} associated to $\rho$ at $w$, i.e., 
\bes 
\mu_\rho(w)= j^*\ltt d^c\rho\wedge(dd^c\rho)^{n-1}\rtt(w)
={\mathcal{M}[\rho](w)\sigma_{_{\!\Om}}(w)},
\ees
where $j:b\Om\rightarrow \C^2$ is the inclusion map, $\sigma_{_{\!\Om}}$ is the surface area measure on $b\Om$, and 
\bes
\mathcal{M}[\rho]=\frac{{-}\operatorname{det}\begin{pmatrix}
    0 & {\partial\rho}/{\partial z_j}\\
    {\partial\rho}/{\partial \bar z_j} & 
    {\partial^2\rho}/{\partial z_j\partial \bar z_k}
\end{pmatrix}_{1\leq j,k\leq n}}{\pi^n\|\nabla \rho\|}.
\ees
Thus, if $\rho$ extends as a $\cont^2$-smooth defining function of $\Om$ on a neighborhood of $\overline\Om$, then $\mu_\rho\ll \sigma_{_{\!\Om}}$, and further, if $\Om$ is strongly pseudoconvex, then $\mu_\rho\approx \sigma_{_{\!\Om}}$. 

When $\Om\in\rt$, we use the fact that $m_\Om$ is $\mathcal C^2$-smooth on $b\Om_+$ and the Monge--Amp{\`e}re measures appearing on the R.H.S. of \eqref{eq:defn MA min} do not charge $\mathcal Z$ to conclude the following. 

\begin{proposition}\label{pr:boundary MA measure}
Let $\Om\in\rt$, and $\rho$ be a $\cont^1$-smooth defining function of $\Om$ that is $\cont^2$-smooth on $b\Om_+$. Then,
$\mu_{_\Om}=\left(\mathds{1}_{b\Om_+}\mathcal{M}[{m_\Om}]\right)\sigma_{_{\!\Om}}\approx \mathds{1}_{b\Om_+}j^*\ltt d^c\rho\wedge(dd^c\rho)\rtt.$
\end{proposition}

We end this section with a connection between the two measures appearing above. This is an analogue of integration using polar coordinates and was proved by Lindholm for strongly convex domains; see \cite[Lemma 3.8]{Li02}.

\begin{proposition}\label{pr:cov}
    Let $\Om\in\rt$. The map $T_\Om:(0,\infty)\times b\Om\rightarrow \C^2\setminus\{0\}$ given by
    \be\label{eq:polar coordinates}
    T_\Om( r,\zeta)=2r\left(\partl{m_\Om}{z_1}(\zeta),\partl{m_\Om} {z_2}(\zeta)\right)
    \ee
    is a homeomorphism from $(0,\infty)\times b\Om$ onto $\C^2\setminus\{0\}$ that restricts to a $\cont^1$-diffeomorphism from $(0,\infty)\times b\Om_+$ onto $\C^2\setminus\mathcal Z$. 
    Moreover, if $\psi$ is a $(dd^cH_\Om)^2$-integrable function, then
    \be\label{eq:cov}
    \int_{\C^2} \psi(z) \ltt dd^c H_\Om\rtt^2(z)=\int_0^\infty\int_{b\Om} \psi\ltt T_\Om \ltt r  ,\zeta\rtt\rtt r\,d\mu_\Om(\zeta) dr.
    \ee
\end{proposition}
    \begin{proof} Let $T^1_\Om(\cdot)=T_\Om(1,\cdot)$. 
 Since $\Om$ is Reinhardt and $m_\Om$ is $1$-homogenous, we have that $\left<T^1_\Om(\zeta),\zeta\right>=1$, for all $\zeta\in b\Om$. Thus, $T^1_\Om$ maps $b\Om$ into $b\Om^*$. Similarly, $T^1_{\Om^*}$ maps $b\Om^*$ into $b\Om$. So, for any $w\in b\Om^*$, $\langle T_\Om^1\ltt T^1_{\Om^*}(w)\rtt,T^1_{\Om^*}(w)\rangle=1=\langle T^1_{\Om^*}(w),w\rangle.$ By \cite[Proposition 18]{Li02}, the Minkowski functional of $\Om^*$ is $H_\Om$. Thus, abusing notation, we have  that $T_{\Om^*}(r,w)=2r\partial H_\Om(w)$. Now, the complex line $\ell_w=\{z\in\C^2:\langle 2\partial H_\Om(w),z\rangle=1\}$ is tangential to $b\Om^*$ at $w$. By the strict convexity of $b\Om^*$, $\ell_w$ cannot intersect $b\Om^*$ at any point other than $w$. Thus, it must be that $T^1_\Om\ltt T^1_{\Om^*}(w)\rtt=w$ for all $w\in b\Om^*$. Similarly, $T^1_{\Om^*}\ltt T^1_\Om(\zeta)\rtt=\zeta$ for all $\zeta\in b\Om$. Thus, $T_\Om^1$ is a continuous bijective map from $b\Om$ onto $b\Om^*$. Moreover, 
the unique complex tangent space at any $\zeta\in b\Om$ is parallel to the axes if and only if $\zeta\in\mathcal Z$. Since $m_\Om$ is $\cont^2$-smooth off of $\mathcal Z$, $T_\Om^1$ maps $b\Om_+$ $\cont^1$-smoothly onto $b\Om^*_+$. This yields the first claim. 

To prove \eqref{eq:cov}, let $T_{+}$ denote the restriction of $T_\Om$ to  $(0,\infty)\times b\Om_+$. Then, following the computation presented in \cite[\S 8]{Be93}, where we set $\rho=m_\Om$ and $\varphi=H_\Om$, we have that
    \bes
T_{+}^*\ltt\mathfrak {H}_\Om d\omega_{std}\rtt= r dr\wedge j^*\ltt d^c m_\Om\wedge dd^c m_\Om\rtt,
    \ees
where $\mathfrak{H}_\Om$ is as in Proposition~\ref{pr:measure equality}. 
The claim now follows from Propositions~ \ref{pr:measure equality} and ~\ref{pr:boundary MA measure}. 
       \end{proof}

\subsection{ Hardy spaces on domains in $\rt$} Recall that $\mathcal H^2(\Om,\mu)$ is defined as the $L^2(b\Om,\mu)$-closure of $\mathcal A(\Om)=\cont(\overline\Om)\cap\hol(\Om)$. Since convex domains have the Mergelyan property, one can also define $\mathcal H^2(\Om,\mu)$ to be the $L^2(b\Om,\mu)$-closure of holomorphic polynomials, as done in \cite{BaLa09}. We comment briefly on what it means for this space to be a reproducing kernel Hilbert space of holomorphic functions on $\Om$.

Suppose $\mathcal A(\Om)$ is strongly admissible in the sense of \cite[Definition~3.4]{GGLV21}, then the map $F|_{b\Om}\mapsto F$ extends from $\mathcal A(\Om)$ to an isometric isomorphism $\mathcal I:\mathcal H^2(\Om,\mu)\rightarrow \mathfrak X(\Om,\mu)$, where $\mathfrak X(\Om,\mu)$ is a reproducing kernel Hilbert space of holomorphic functions on $\Om$ containing $\mathcal A(\Om)$. This gives a unique function $s:\Om\times \text{supp}\:\mu\rightarrow \C$, with $s(\cdot,w)\in\mathfrak X(\Om,\nu)$ for all $w\in\text{supp}\:\mu$ and $s(z,\cdot)\in\mathcal H^2(\Om,\mu)$ for all $z\in\Om$, such that 
    \bes        (\mathcal If)(z)=\int_{b\Om}f(w)s(z,w)d\mu(w),\quad z\in\Om, f\in\mathcal H^2(\Om,\mu).    \ees
This extends to a well-defined bounded transformation $\mathbb S:L^2(b\Om,\mu)\rightarrow\mathfrak X(\Om,\mu)$ that ``produces'' and ``reproduces'' holomorphic functions:
\bes
    \mathbb Sf(z)=\int_{b\Om}f(w)s(z,w)d\mu(w),\quad z\in\Om.
\ees
Further, a projection operator $\mathbb S_b:L^2(b\Om,\mu)\rightarrow\mathcal H^2(\Om,\mu)$ can be defined as follows. 
\bes
\mathbb S_bf= \mathcal I^{-1}(\mathbb Sf),\quad f\in L^2(b\Om,\mu).    \ees

In the case when $\Om$ is $\cont^2$-smooth and convex, and $\mu=\sigma$ is the surface area measure, $\mathcal A$ is indeed strongly admissible, $\mathbb S$ and $\mathbb S_b$ are bounded operators on $L^2(b\Om,\sigma)$,  $\mathbb S_bf(\zeta)=\lim_{r\rightarrow 1^-}\mathbb Sf(r\zeta)$ for $\sigma_\Om$-a.e. $\zeta\in b\Om$ and $f\in L^2(b\Om,\sigma)$, and $\mathbb S_b$ is in fact the (orthogonal) Szeg{\H o} projection from $L^2(b\Om,\sigma)$ onto $\mathcal H^2(\Om,\sigma)$; see \cite{AGV14}. For a general convex domain $\Om\subset\Cn$ and $\mu=\mu_\rho$ as in \eqref{E:ma}, the obvious candidate for $\mathfrak X(\Om,\mu)$ is the Poletsky--Stessin Hardy space on $\Om$ corresponding to $\rho$; see \cite{PoSt08}. However, the strong admissibility of $\mathcal A$ or the existence of boundary values for functions in $\mathfrak X(\Om,\mu)$ is not known in this general setting. For $\Om\in\rt$, we give an explicit description of $\mathfrak X(\Om,\mu)$ using series expansions in Section~\ref{subsec:series expansion}.

\section{Some technical tools}\label{sub:tech tools}

Our computations are facilitated by a special parametrization of domains in $\rt$. We describe the relevant function spaces and operators in terms of this parametrization in this section.
\subsection{\textbf{An alternate parametrization}}\label{subsec:param}  Let $\Om$ be a bounded convex Reinhardt domain in $\C^2$ that is strongly convex away from $\mathcal Z$. Let $\phi,b_1$  {and} $b_2$ be as granted by (\ref{eq:graph param gamma}). Since $\Om$ is Reinhardt, any parameterization of the curve $\gamma_\Om$  induces a parameterization of $b\Om$. Thus, the parameterization of $\gamma_\Om$ given in (\ref{eq:graph param gamma}) induces a parameterization of $b\Om$. However, it is fruitful to consider a different parameterization of $\gamma_\Om$ that was introduced in \cite[\S 2]{BaLa09}. Let 
\bea\label{eq:param s}
s(r)=
\begin{cases}
\frac{-r\phi^\prime (r)}{\phi(r)-r\phi^\prime(r)},& r\in (0,b_1),\\
0,& r=0,\\
1,& r=b_1.
\end{cases}
\eea
 Then, due to properties $(i)$, $(ii)$ and $(iii')$ of $\phi$, $s:[0,b_1]\rightarrow[0,1]$ is a continuous monotone function that is $\cont^1$-smooth on $\ltt 0,b_1\rtt$; see the discussion after Prop. 8 in \cite{BaLa09}. Denote $s^{-1}$ by $r_1$. Then, we have that
\be\label{eq:full parameterization}
b\Om=\left\{\left(r_1(s)e^{i\theta_1},r_2(s)e^{i\theta_2}\right):0\leq s\leq 1,(\theta_1,\theta_2)\in [0,2\pi)\times[0,2\pi)\right\},
\ee
where $r_2=\phi\circ r_1$. The functions, $r_1$ and $r_2$ are related to the exponent function $p$ from (\ref{eq:reln between kappa an p}) as follows. Let $\check{p}:(0,1)\rightarrow (1,\infty)$ be given by $\check{p}(s)=p\ltt r_1(s),r_2(s)\rtt$. Then, for $s\in(0,1)$,
\begin{align}
&r_1(s)=b_1\exp\left(-\int_s^1\frac{dt}{t{\check{p}}(t)}\right),\label{eq:r_1 in s} \\
&r_2(s)=b_2\exp\left(-\int_0^s\frac{dt}{(1-t)\check{p}(t)}\right)\label{eq:r_2 in s}.
\end{align}
Moreover,
\begin{align}\label{eq:Two integral condition}
    \int_0^1 \frac{dt}{t{\check{p}}(t)} = 
    \int_0^1 \frac{dt}{(1-t){\check{p}}(t)} =\infty.
\end{align}
If $\Om\in\rt$, then in addition to \eqref{eq:Two integral condition},
\begin{align}\label{eq:Four integral condition}
    \int_0^1 \left(1-\frac{1}{\check{p}(t)}\right) \frac{dt}{t} =  
    \int_0^1 \left(1-\frac{1}{\check{p}(t)}\right) \frac{dt}{(1-t)} = \infty.
\end{align}
Conversely, by \cite[Theorem 9]{BaLa09}, each tuple $\ltt\check{p},b_1,b_2\rtt$, where $b_1,b_2>0$ and $\check{p}:(0,1)\rightarrow(1,\infty)$ is a continuous function satisfying \eqref{eq:Two integral condition} and \eqref{eq:Four integral condition}, determines a unique $\Om\in\rt$.

 \begin{example}\label{ex:rprime} Domains in $\mathcal R'$ (see \eqref{D:newclass}) are characterized by tuples of the form $(\check{p},b_1,b_2)$, where $b_1,b_2$ are positive constants and $\check{p}:[0,1]\rightarrow(1,\infty)$ is a continuous function which satisfies \eqref{eq:Two integral condition} and \eqref{eq:Four integral condition}. 
By \cite[Theorem 19]{BaLa09}, the additional condition
\be\label{eq:rcond}
    \int_0^1\ltt\frac{1}{\check{p}(s)}-\frac{1}{\check{p}(0)}\rtt\frac{ds}{s}<\infty \quad \text{and}\quad  \int_0^1\ltt\frac{1}{\check{p}(s)}-\frac{1}{\check{p}(1)}\rtt\frac{ds}{1-s}<\infty
\ee
characterizes domains in $\mathcal R$. Thus, $\mathcal R\subseteq\mathcal R'$. The tuple $(\check{p},1,1)$, where 
\[
\check{p}(s)=2+\frac{1}{\log\ltt\frac{10}{s}\rtt},\quad s\in[0,1],
\]
satisfies \eqref{eq:Two integral condition} and \eqref{eq:Four integral condition}, but does not satisfy the integral conditions in \eqref{eq:rcond}. Hence the corresponding domain is in $\mathcal R'$, but not in $\mathcal{R}$. 
\end{example}

The parametrization discussed above also yields a quick proof of Lemma (\ref{le:curvature and p}).

\begin{proof}[Proof of Lemma~\ref{le:curvature and p}] Let $\zeta\in b\Om_+$. By \eqref{eq:full parameterization}, we have that
$\zeta=\ltt r_1(s)e^{i\theta_1},r_2(s)e^{i\theta_2}\rtt$ for some $\ltt s,\theta_1,\theta_2\rtt\in (0,1)\times[0,2\pi)^2$. From \cite[(4.5)]{BaLa09}, the three principal curvatures at $\zeta$ are given by
\begin{align*}
\kappa_1\ltt\zeta\rtt&=\frac{s}{{r_1(s)}^2\sqrt{\ltt \frac{s}{r_1(s)}\rtt^2+\ltt \frac{1-s}{r_2(s)}\rtt^2}},\\
\kappa_2\ltt\zeta\rtt&=\frac{1-s}{{r_2(s)}^2\sqrt{\ltt \frac{s}{r_1(s)}\rtt^2+\ltt \frac{1-s}{r_2(s)}\rtt^2}},\\
\kappa_3\ltt\zeta\rtt&=\ltt\check{p}(s)-1 \rtt\frac{s(1-s)}{{r_1(s)}^2 {r_2(s)}^2}\ltt\ltt \frac{s}{r_1(s)}\rtt^2+\ltt \frac{1-s}{r_2(s)}\rtt^2\rtt^{-\frac{3}{2}}.
\end{align*}
  Since $p$ is rotation-invariant in each coordinate, $p\ltt\zeta\rtt=\check{p}(s)$. Thus,
   \be\label{eq:curvature p relation wrt s}
p(\zeta)=1+\ltt\frac{\kappa_3}{\kappa_1\kappa_2}\ltt\zeta\rtt\times \sqrt{\ltt \frac{s}{r_1(s)}\rtt^2+\ltt \frac{1-s}{r_2(s)}\rtt^2}
    \rtt.
    \ee
Let $\phi$ be as in \eqref{eq:graph param gamma}, and \be\label{eq:B-L defining function}
\widetilde\rho(z_1,z_2)=|z_2|-\phi(|z_1|),\qquad (z_1,z_2)\in \overline{\mathbb B^1(b_1)}\times\C.
\ee
Note that $\widetilde\rho$ is a $\mathcal{C}^2$-smooth defining function of $\Om$ near any $\zeta\in b\Om_+$. By \cite[(2.23)]{BaLa09}, $\phi^{'}\ltt{r_1(s)}\rtt=-\frac{sr_2(s)}{(1-s)r_1(s)}$. 
Thus, for this specific choice of defining function, we obtain that
    \bes
    \frac{\Vert\nabla \widetilde\rho(\zeta)\Vert}{2\operatorname{Re}\left<\partial{\widetilde\rho}(\zeta),\zeta\right>}=\frac{\sqrt{1+\phi^{'}\ltt{r_1(s)}\rtt^2}}{r_2(s)-r_1(s) \phi^{'}\ltt{r_1(s)}\rtt}=\sqrt{\ltt \frac{s}{r_1(s)}\rtt^2+\ltt \frac{1-s}{r_2(s)}\rtt^2}.
    \ees
    Combining this with \eqref{eq:curvature p relation wrt s}, we obtain \eqref{eq:B-L defining function} for this specific choice of defining. However, $\frac{\Vert\nabla \rho(\zeta)\Vert}{2\operatorname{Re}\left<\partial{\rho}(\zeta),\zeta\right>}$ is independent of the choice of $\cont^1$-defining function near $\zeta$. Our proof is complete.
    
\end{proof}

 The above parameterization transforms particularly well under the action of taking the dual complement of $\Om$; see \cite[\S 7]{BaLa09}. Let $\ltt\check{p}^{*},b_1^{*},b_2^{*}\rtt$ be the tuple corresponding to $\Om^*$. Then \begin{align}\label{eq:pp*bb*}
\frac{1}{\check{p}}+\frac{1}{\check{p}^*}\equiv 1\,\,\text{and}\,\,
b_1^*b_1= b_2^{*}b_2=1.
\end{align}
Consequently, $
    b\Om^*=\left\{\left(r_1^*(s)e^{i\theta_1},r_2^*(s)e^{i\theta_2}\right):0\leq s\leq 1,(\theta_1,\theta_2)\in [0,2\pi)\times[0,2\pi)\right\}$,
where
\begin{align}\label{eq:r_1^* r_2^* in s}
r_1^*(s)=\frac{s}{r_1(s)}\,\,\text{and}\,\,
r_2^*(s)=\frac{(1-s)}{r_2(s)}, \quad \text{for all }s\in[0,1].
\end{align}
To compute the map $T_\Om:(0,\infty)\times b\Om\rightarrow\C^2$ as in \eqref{eq:polar coordinates}, observe that as $\Om$ is Reinhardt and $m_\Om$ is $1$-homogeneous $2\left<\partial m_\Om(\zeta),\zeta\right>=1,\,\forall \zeta\in b\Om$; see \cite[Proposition 18]{Li02}. Thus, for any $\zeta\in b\Om_+$,
\bes
2\partl{m_\Om}{z_l}(\zeta)=\frac{\partl{m_\Om}{z_l}(\zeta)}{\left<\partial m_\Om(\zeta),\zeta\right>}=\frac{\partl{\widetilde\rho}{z_l}(\zeta)}{\left<\partial \widetilde\rho(\zeta),\zeta\right>},\quad  l=1,2,
\ees
where $\widetilde\rho$ is as in \eqref{eq:B-L defining function}, and the last equality holds since both $m_\Om-1$ and $\Tilde{\rho}$ are local defining functions of $b\Om_+$. Thus, for $\zeta=\ltt r_1(s)e^{i\theta_1},r_2(s)e^{i\theta_2}\rtt$,
\be\label{eq:TOmega}
T_\Om(\zeta)=r\ltt\frac{s}{r_1(s)}e^{-i\theta_1},\frac{(1-s)}{r_2(s)}e^{-i\theta_2}\rtt=r \ltt r_1^*(s)e^{-i\theta_1},r_2^*(s)e^{-i\theta_2}\rtt.
\ee

\subsection{Series expansions for Hardy and Bergman-space functions}\label{subsec:series expansion}
It is well-known that the functions in the Hardy space of the unit disc are precisely those whose Fourier expansions have no negative Fourier coefficients. Due to the toric symmetry and convexity of domains in $\rt$, an analogous statement holds for functions in Hardy spaces on such domains. To see this, let $\vartheta:(0,1)\times [0,2\pi)^2\rightarrow b\Om_+$ denote the parametrization given by $(s,\theta_1,\theta_2)\mapsto \left(r_1(s)e^{i\theta_1},r_2(s)e^{i\theta_2}\right)$. Then, by Proposition~\ref{pr:boundary MA measure} and \cite[(4.2)]{BaLa09},
    \be\label{eq:boundary MA wrt param}
        \left(\vartheta^*\mu_\Om\right)(s,\theta_1,\theta_2)=\frac{1}{16\pi^2}\:ds\,d\theta_1 d\theta_2.
    \ee
 Thus, there is a normed space isomorphism between $L^2(b\Om,\mu_\Om)$ and $L^2((0,1)\times [0,2\pi)^2,ds\:d\theta_1d\theta_2)$. In view of this, for any $f\in L^2(b\Om,\mu_\Om)$, we abbreviate $f(r_1(s)e^{i\theta_1},r_2(s)e^{i\theta_2})$ to $f(s, \theta_1,\theta_2)$, and identify the two spaces without further comment.   

\begin{lemma}\label{pr:fourier series of Hardy space} Let $\Om\in\rt$. Then, $f\in\mathcal{H}^2\ltt \Om,\mu_\Om\rtt$ if and only if 
\be\label{eq:fourier of f}
f(s,\theta_1,\theta_2)=\sum_{(m_1,m_2)\in \N^2}a_{m_1,m_2} {r_1(s)}^{m_1}{r_2(s)}^{m_2} e^{i({\theta_1 {m_1}+\theta_2 {m_2}})} \quad \left(\text{in }L^2(b\Om,\mu_\Om)\right)
\ee
for some sequence of complex numbers $(a_{m_1,m_2})_{\N^2}$ satisfying
\be\label{eq:amn condition}
\sum_{(m_1,m_2)\in \N^2}|a_{m_1,m_2}|^2\left(\int_0^1r_1(s)^{2m_1} r_2(s)^{2m_2}ds\right)<\infty.
\ee
Moreover, if $f$ and $(a_{m_1,m_2})_{\N^2}$ are related as in \eqref{eq:fourier of f}, then
\be\label{eq:norm in hardy space}
\|f\|^2_{\mathcal H^2(\Om,\mu_\Om)}\approx\sum_{(m_1,m_2)\in \N^2}|a_{m_1,m_2}|^2\left(\int_0^1r_1(s)^{2m_1} r_2(s)^{2m_2}ds\right).
\ee
\end{lemma}
\begin{proof} Let $\Om\in\rt$, and $f\in\mathcal H^2(\Om,\mu_\Om)$. For each $s\in(0,1)$, let $f_s:[0,2\pi)\times[0,2\pi)\rightarrow \C$ be given by
\bes
f_s\ltt\theta_1,\theta_2\rtt=f\ltt s,\theta_1,\theta_2\rtt.
\ees
Then, $f_s\in L^2\ltt d\theta_1d\theta_2\rtt$ for almost every $s\in (0,1).$
Thus, for almost every $s\in(0,1)$,
\be\label{eq:fourier fs}
f_s(\theta_1,\theta_2)=\sum_{(m_1,m_2)\in \mathbb{Z}^2} \widehat{f_s}(m_1,m_2) e^{i(\theta_1 m_1+\theta_2 m_2 )}\quad \text{in } L^2\ltt d\theta_1d\theta_2\rtt,
\ee
where $\widehat{f_s}(m_1,m_2)$, $(m_1,m_2)\in\Z^2$, are the Fourier coefficients of $f_s$. 

Since holomorphic polynomials are dense in $\mathcal H^2(\Om,\mu_\Om)$, there is a sequence of polynomials 
\bes
p_k\ltt z_1,z_2\rtt=\sum_{m_1=0}^{j_1(k)}\sum_{m_2=0}^{j_2(k)} a_{m_1,m_2}^{(k)} z_1^{m_1} z_2^{m_2},\quad k\in\N,
\ees
that converges to $f$ in $L^2\ltt b\Om,\mu_\Om\rtt$. For each $s\in(0,1)$, we have that
\beas
\widehat {\left(p_k\right)_s}( m_1,m_2)=
\begin{cases}
a_{m_1,m_2}^{(k)}r_1(s)^{m_1} r_2(s)^{m_2},\quad & 0\leq m_1\leq j_1(k), 0\leq m_2\leq j_2(k),\\
0,& \text{otherwise}.
\end{cases}
\eeas
By Parseval's formula and Fubini's theorem, we obtain that
\bes
\int_0^1\int_0^{2\pi}\int_0^{2\pi}|p_k-f|^2 d\theta_1 d\theta_2 \,ds\approx \sum_{(m_1,m_2)\in\Z^2} \int_0^1 \left|\widehat{\left(p_k\right)_s}(m_1,m_2)-\widehat{f_s}(m_1,m_2)\right|^2 ds.
\ees
Since the sequence on the left-hand side converges to $0$ as $k\rightarrow \infty$,
\bea\label{eq:conv fourier for geq 0}
\lim_{k\rightarrow\infty}\int_0^1 \left|\widehat{\ltt p_k\rtt_{s}}(m_1,m_2)-\widehat{f_s}(m_1,m_2)\right|^2 ds=0,\quad \forall (m_1,m_2)\in \Z^2.
\eea
In particular, for each $(m_1,m_2)\in\Z^2$, there exists some $a_{m_1,m_2}\in\C$ such that
    \begin{align*}
       \lim_{k\rightarrow\infty}\int_0^1 \widehat{\ltt p_k\rtt_{s}}(m_1,m_2)\: ds &=\left( \int_0^1 r_1(s)^{m_1}r_2(s)^{m_2}ds\right)\lim_{k\rightarrow\infty}a_{m_1,m_2}^{(k)}\\
    &=\left( \int_0^1 r_1(s)^{m_1}r_2(s)^{m_2}ds\right)a_{m_1,m_2}.
    \end{align*}
Since $\ltt a_{m_1,m_2}^{(k)}\rtt_{k\in\N}$ is the constant zero sequence whenever $\min\{m_1,m_2\}<0$, we have that $a_{m_1,m_2}=0$ whenever $\min\{m_1,m_2\}<0$. Now, by the fact that $L^2$-convergence implies a.e. pointwise convergence of a subsequence, we obtain that, for a.e. $s\in (0,1)$,
\beas
\widehat{f_s}(m_1,m_2)=\begin{cases}a_{m_1,m_2} r_1(s)^{m_1} r_2(s)^{m_2},\quad &\min\{m_1,m_2\}\geq0,\\
0,& \min\{m_1,m_2\}<0.
\end{cases}
\eeas
This yields that \eqref{eq:fourier of f}, \eqref{eq:amn condition}, and \eqref{eq:norm in hardy space} hold for any $f\in\mathcal H^2(\Om,\mu_\Om)$.

Conversely, if $(a_{m_1,m_2})_{\N^2}$ is a sequence of complex numbers satisfying \eqref{eq:amn condition}, then the function $f$ defined as in \eqref{eq:fourier of f} is in $L^2(b\Om,\mu_\Om)$, and is approximable therein by holomorphic polynomials of the form  $p_k(z_1,z_2)=\sum_{m_1,m_2=0}^ka_{m_1,m_2}z_1^{m_1} z_2^{m_2}$, $k\in\N$. Thus, $f\in\mathcal H^2(\Om,\mu_\Om)$.
\end{proof}
It is now possible to write the power series expansion of the Laplace transform of a Hardy-space function.

\begin{lemma}\label{le:power series of laplace} Let $\Om\in\rt$, and $f\in\mathcal H^2(\Om,\mu_\Om)$ admit the expansion \eqref{eq:fourier of f} in $L^2(b\Om,\mu_\Om)$. Then, the Laplace transform of $f$ is an entire function on $\C^2$ with power series expansion
\be\label{eq:power series of FL}
\mathcal{L}\ltt f \rtt(z_1,z_2)=\sum_{(m_1,m_2)\in \N^2} t_{m_{1},m_{2}} z_1^{m_1} z_2^{m_2},
\ee
where the coefficients $t_{{m_1},{m_2}}$ are given by
\be\label{eq:coeff of power series}
t_{m_1,m_2}=\frac{1}{4}\frac{\overline{a_{m_1,m_2}}}{m_1!m_2!} \ltt\int_0^1 r_1(s)^{2m_1}r_2(s)^{2m_2}ds\rtt.
\ee 
\end{lemma}
\begin{proof} By the Cauchy--Schwarz inequality, for any compact set $K\subset\C^2$, there is a $C_K>0$ such that
    \bes
\Vert \mathcal{L}(f)\Vert^2_{\cont(K)}\leq
C_K \Vert f\Vert^2_{\mu_\Om},\quad f\in \mathcal H^2(\Om,\mu_\Om).
\ees
Thus, it suffices to establish the claim for polynomials of the form 
    $
    p_k\ltt z_1,z_2\rtt=\sum\limits_{m_1,m_2=0}^k a_{m_1,m_2}z_1^{m_1}\,z_2^{m_2}.
    $
Expanding $e^{\langle.,z\rangle}$ in terms of its power series and using the parameterization of $b\Om$ discussed in Section~\ref{subsec:param}, we have that
    \begin{align*}    
    \mathcal{L}\ltt p_k\rtt\ltt z_1,z_2\rtt&=\int_{b\Om} \overline{{p_k}(\zeta)} e^{\langle\zeta,z\rangle} d\mu_\Om(\zeta)\\
&=\frac{1}{16\pi^2}\int_0^1\int_0^{2\pi}\int_0^{2\pi}\ltt {\sum_{m_1,m_2=0}^k\overline{a_{m_1,m_2}} {r_1(s)}^{m_1}{r_2(s)}^{m_2} e^{-i({\theta_1 m_1+\theta_2 m_2})}}\rtt \\
    &\hspace{4cm}\times\left(\sum_{(k_1,k_2)\in \N^2}\frac{r_1(s)^{k_1}r_2(s)^{k_2}e^{i(k_1\theta_1+k_2\theta_2)}}{k_1!k_2!} z_1^{k_1}z_2^{k_2}\right) d\theta_1 d\theta_2 ds\\
&=\frac{1}{4} {\sum_{m_1,m_2=0}^k}\frac{\overline{a_{m_1,m_2}}}{m_1!m_2!}\left(\int_0^1r_1(s)^{2m_1} r_2(s)^{2m_2}\,ds\right) z_1^{m_1} z_2^{m_2},
\end{align*}
where the last equality is obtained by integrating over the $\theta_1$ and $\theta_2$ variables. 
\end{proof}

\begin{lemma}\label{le:series Bergman space} Let $\Om\in\rt$. Let $F$ be an entire function on $\C^2$ with power series expansion
\be\label{eq:series of F}
F(z_1,z_2)=\sum_{\ltt m_1,m_2\rtt\in \N^2}\beta_{m_1,m_2}z_1^{m_1}z_2^{m_2},
\ee
for some sequence of complex numbers $(\beta_{m_1,m_2})_{\N^2}$. Then,  $F\in\mathcal{A}^2(\C^2,\nu_\Om)$ if and only if 
\be\label{eq:bmn condition}
\sum_{\ltt m_1,m_2\rtt\in\N^2}\vert\beta_{m_1,m_2}\vert ^2(m_1+m_2+1)!^2 \left(\int_0^1 (r_1^{*}(s))^{2m_1} (r_2^{*}(s))^{2m_2} ds\rtt<\infty. 
\ee
Moreover, if $F\in\mathcal{A}^2(\C^2,\nu_\Om)$, then
\be\label{eq:norm in bergman space}
\|F\|^2_{\nu_\Om}\approx\sum_{\ltt m_1,m_2\rtt\in\N^2}\vert\beta_{m_1,m_2}\vert ^2(m_1+m_2+1)!^2 \left(\int_0^1 (r_1^{*}(s))^{2m_1} (r_2^{*}(s))^{2m_2} ds\rtt.
\ee
\end{lemma}
\begin{proof}{
 First, since $B\ltt0,r_1\rtt\Subset \Om\Subset B\ltt0,r_2\rtt$, for some $r_1,r_2>0$, we have that $r_1\|z\|\leq H_\Om(z)\leq r_2\|z\|$, i.e., $H_\Om(z)\approx \|z\|$. Now, we apply \eqref{eq:cov}  to $\psi=\vert F\vert^2 e^{-2H_\Om}H_\Om^{3/2}$, noting that $H_\Om(T_\Om(r,\zeta))=r$. Then, parametrizing $b\Om$ as in Section~\ref{subsec:param}, expressing $T_\Om$ and $\mu_\Om$ in terms of this parametrization as in \eqref{eq:TOmega} and  \eqref{eq:boundary MA wrt param}, we obtain that
\begin{align*}
\Vert F\Vert_{\nu_\Om}^2&=\int_{\C^2}\left\vert F(z)\right\vert^2 e^{-2H_\Om(z)}\|z\|^\frac{3}{2}\left(dd^cH_\Om\right)^2(z)\\
&\approx
\int_{\C^2}\left\vert F(z)\right\vert^2 e^{-2H_\Om(z)}H_\Om(z)^\frac{3}{2}\left(dd^cH_\Om\right)^2(z)\\
&\approx 
\int_0^\infty\int_{b\Om}\left\vert F\ltt T_\Om\ltt r,\zeta\rtt\rtt\right\vert^2 e^{-2r} r^{\frac{5}{2}} d\mu_\Om(\zeta) dr
\\
&\approx  \int_0^\infty\int_0^1\,\int_0^{2\pi}\,\int_0^{2\pi}\left\vert F(rr_1^*(s)e^{-i\theta_1},rr_2^*(s)e^{-i\theta_2})\right\vert^2 e^{-2r} r^{\frac{5}{2}} d\theta_1 d\theta_2 ds dr\\
&= \int_0^\infty
\left(4\pi^2\sum_{(m_1,m_2)\in \N^2}r^{2(m_1+m_2)}|\beta_{m_1,m_2}|^2\int_0^1{r_1^*(s)}^{2m_1} {r_2^*(s)}^{2m_2}ds\right)e^{-2r} r^{\frac{5}{2}} dr\\
& \approx \sum_{(m_1,m_2)\in \N^2}|\beta_{m_1,m_2}|^2(m_1+m_2+1)!^2 \left(\int_0^1 r_1^{*}(s)^{2m_1} r_2^{*}(s)^{2m_2}\,ds\right),
\end{align*}
where, in the final step, we use the following consequence of Stirling's approximation:
\be\label{eq:estimate}
    \left(\int_0^\infty r^{2k+1} e^{-2r} r^{-\frac{1}{2}}\,dr\right)\approx (k!)^2.
    \ee}
    From this comparison, the lemma immediately follows.
\end{proof}
\subsection{The Leray transform.} We briefly recall the construction and properties of the Leray transform as given in \cite[\S3]{BaLa09}. 
Given $\Om\in\rt$,
    \bes
        \Le(f)(z)=\int_{b\Om}f(\zeta)L(\zeta,z), \quad z\in\Om,
    \ees
where, the (Leray) kernel is given by
\[L(\zeta,z)=
\frac{d\mu_\Om(\zeta)}{\langle\partial m_\Om(\zeta),\zeta-z\rangle^2}=\mathds{1}_{b\Om_+}(\zeta)\frac{j^*\ltt d^cm_\Om\wedge dd^cm_\Om\rtt(\zeta)}{\langle\partial m_\Om(\zeta),\zeta-z\rangle^2},
\]
by Proposition~\ref{pr:boundary MA measure}. Since the right-hand side is independent of the choice of $\cont^2$-smooth defining function of $b\Om_+$, we replace $m_\Om$ by $\rho$ from 
 \eqref{eq:B-L defining function}, and use the parametrization given in \eqref{eq:full parameterization}. Thus, for $\zeta=\ltt r_1(s)e^{i\theta_1},r_2(s)e^{i\theta_2}\rtt\in b\Om$ and $z=(z_1,z_2)\in\Om$,
\begin{align*}
L(\zeta,z)=\frac{ds\,d\theta_1\,d\theta_2}{4\pi^2}\sum\limits_{m_1,m_2=0}^\infty \frac{(m_1+m_2+1)!}{m_1!m_2!}r_1^*(s)^{m_1} r_2^*(s)^{m_2} z_1^{m_1}z_2^{m_2}e^{-i\ltt m_1\theta_1+m_2\theta_2\rtt},
\end{align*}
where the above series converges uniformly in $z$ on compact subsets of $\Om$.

Following \cite{BaLa09}, a function of the form $g(s)e^{i\ltt m_1\theta_1+m_2\theta_2\rtt}$ is called an $(m_1,m_2)$-monomial, and the space of all $(m_1,m_2)$-monomials that are in $L^2(b\Om,\mu_\Om)$ is denoted by $L^2_{m_1,m_2}(b\Om,\mu_\Om)$. The spaces $L^2_{m_1,m_2}\ltt b\Om,\mu_\Om\rtt$, $(m_1,m_2)\in\Z\times\Z$, are mutually orthogonal, and
\be\label{eq:decomp}L^2\ltt b\Om,\mu_\Om\rtt=\underset{(m_1,m_2)\in\Z^2}{\bigoplus} L^2_{m_1,m_2}\ltt b\Om,\mu_\Om\rtt.\ee
By direct computation, it is shown in \cite[Corollary 23]{BaLa09} that $\Le$ and $\Le_b$ are well-defined on each $L^2_{m_1,m_2}(b\Om,\mu_\Om)$, and therefore densely-defined on $L^2(b\Om,\mu_\Om)$. Furthermore, $\Le_b$ reproduces (the restrictions to $b\Om$ of) holomorphic polynomials, and annihilates any $(m_1,m_2)$-monomial with $\min\{m_1,m_2\}<0$. Finally, since the measure $\mu_\Om$ is admissible in the sense of \cite[Definition 28]{BaLa09}, we have the following result as a combination of Proposition 26, Theorem 30 and Proposition 32 in \cite{BaLa09}.   %

\begin{lemma}\label{le:Leray bdd criteria rt} Let $\Om\in\rt$. Given $(m_1,m_2)\in\Z\times\Z$, let $\Le_{m_1,m_2}$ denote the restriction of $\Le_b$ to $L^2_{m_1,m_2}\ltt b\Om,\mu_\Om\rtt$.
\begin{enumerate}
    \item If $\min\{m_1,m_2\}\geq0$, then  $\Le_{m_1,m_2}$ is a bounded rank-one projection operator on $ L^2_{m_1,m_2}\ltt b\Om,\mu_\Om\rtt$, with 
\begin{align}\label{eq:Lmn}
\Vert\Le_{m_1,m_2}\Vert^2_{\mu_{\Om}}=\gamma_{m_1,m_2}^2\int_0^1 r_1(s)^{2m_1}r_2(s)^{2m_2} ds\int_0^1 {r_1^{*}(s)}^{2m_1}{r_2^{*}(s)}^{2m_2}{ds},
  \end{align}
where $\gamma_{m_1,m_2}=\frac{\ltt m_1+m_2+1 \rtt !
    }{m_1!m_2!}$.
    \item The densely-defined operator $\Le_b$ extends as a bounded operator on $L^2\ltt b\Om, \mu_\Om\rtt$ if and only if the quantities $\Vert\Le_{m_1,m_2}\Vert^2_{\mu_{\Om}}$ are uniformly bounded for $(m_1,m_2)\in\N\times\N$. 
    \item  If $\Le_b$ extends as a bounded operator on $L^2\ltt b\Om,\mu_\Om\rtt$, it extends as a bounded projection operator from $L^2\ltt b\Om,\mu_\Om\rtt$ onto $\mathcal{H}^2\ltt \Om,\mu_\Om\rtt$.
\end{enumerate}
\end{lemma}

\section{Proof Of Theorem ~\ref{th:Leray}}\label{sub:proof PW part 2}
For the first claim, we essentially mimic the proof of Theorem 1 in \cite{BaLa09}. The main technical tool used therein is Theorem 45, which gives the asymptotic values of the operator-norms of $\Le_{m_1,m_2}$ along special sequences in $\N\times\N$. Since we assume a weaker condition on the exponent function $p$
 (but our measure is more special), we require a modification of Theorem 45 in \cite{BaLa09}. This modification is stated as Lemma~\ref{le:key prop for Leray} below. For the second claim, we produce a sequence in $\N\times\N$ for which the corresponding operator-norms of $\Le_{m_1,m_2}$ are unbounded, and then invoke Lemma \ref{le:Leray bdd criteria rt}.

\subsection{Proof of Theorem~\ref{th:Leray} (i)} Recall, from \eqref{D:newclass}, that the exponent function ${p}$ associated to $\Om$ extends continuously from $b\Om$ to $(1,\infty)$. Since $s:[0,1]\rightarrow b\Om$ is a continuous function that assumes the values $0$ and $1$ at the axes, $\check{p}=p\circ s$ extends as a continuous function from $[0,1]$ to $(1,\infty)$. 
\begin{lemma}\label{le:key prop for Leray}
Let $\Om\in\mathcal{R}^{'}$. Let $({m_1}_{j},{m_2}_{j})_{j\in\N}\subset\N\times\N$ be a sequence such that \linebreak $\lim_{j\rightarrow\infty}\max\{{m_1}_{j},{m_2}_{j}\}=\infty$. Then, 
\bes    \lim_{j\rightarrow\infty}\|\Le_{{m_1}_{j},{m_2}_{j}}\|^2_{\mu_\Om}\ \text{exists} 
\ees
in each of the following cases: 
\begin{enumerate}[label=(\alph*)]
    \item 
    if $\lim_{j\rightarrow\infty}min\{{m_1}_{j},{m_2}_{j}\}=\infty$, and there is an $x\in[0,\infty]$ such that $\lim_{j\rightarrow\infty}\frac{{m_1}_{j}}{{m_2}_{j}}=x$;
 
  \item
  if ${m_1}_{j}$ is independent of $j$, and $\lim_{j\rightarrow\infty}{m_2}_j= \infty$; 
 \item
  if ${m_2}_{j}$ is independent of $j$, and $\lim_{j\rightarrow\infty}{m_1}_j= \infty$. 
\end{enumerate}
\end{lemma}

Assuming the above lemma, we complete the proof of Theorem~\ref{th:Leray} (i). By Lemma~\ref{le:Leray bdd criteria rt}, it suffices to show that $\|\Le_{m_1,m_2}\|_{\mu_\Om}$ is uniformly bounded for $(m_1,m_2)\in\N\times\N$. If the latter doesn't hold, there exists a sequence $({m_1}_j,{m_2}_j)_{j\in\N}\subset \N\times\N$ with $\lim_{j\rightarrow\infty}\max\{{m_1}_j,{m_2}_j\}=\infty$ such that $\|\Le_{{m_1}_j,{m_2}_j}\|_{\mu_\Om}$ is unbounded along any subsequence of $(m_{1j},m_{2j})_{j\in\N}$. Such a sequence always admits a subsequence, which we still denote by $({m_1}_j,{m_2}_j)_{j\in\N}$, that satisfies one of the conditions $(a)$, $(b)$ or $(c)$ in Lemma~\ref{le:key prop for Leray}. But, by Lemma~\ref{le:key prop for Leray},  $\|\Le_{{m_1}_j,{m_2}_j}\|_{\mu_\Om}$ converges to a positive constant, which is a contradiction. Thus, the quantities $\|\Le_{{m_1},{m_2}}\|_{\mu_\Om}$ are uniformly bounded, and $\Le_b$ extends as a bounded projection operator from $L^2\ltt b\Om,\mu_\Om\rtt$ onto $\mathcal{H}^2\ltt \Om,\mu_\Om\rtt$. But for the proof of Lemma~\ref{le:key prop for Leray}, we are done.
\smallskip

\noindent{\em Proof of Lemma~\ref{le:key prop for Leray}.} In the case of $(a)$, the proof is exactly the same as that of \cite[Theorem 45, part (a)]{BaLa09} with $h_k=\omega^k$ set as $1$; see\cite[Page 2807]{BaLa09}. The proof therein only depends on the continuous extendability of $\check{p}$ to $[0,1]$. See Section~\ref{SS:proof1.4ii} for a summary of this proof.  
 
We prove the result in the case of $(b)$. By switching the roles of $r_1$ and $r_2$ in the following argument, we also obtain the result in the case of $(c)$. Suppose $m_{1j}=m_0\in\N$ for all $j\in \N$. Let $\check{p}_0=\lim_{s\rightarrow 0}\check{p}(s)$. 
By \eqref{eq:r_1 in s},
    \be\label{eq:r1}
r_1(s)=s^{1/\ptt_0} \tau(s),\quad s\in(0,1),
    \ee
   where
$\tau(s)=b_1\exp\ltt{\int_s^1\ltt{\frac{1}{\ptt_0}-\frac{1}{\ptt(t)}}\rtt\frac{dt}{t}}\rtt.$
Then, for any $x>0$,
\be\label{eq:reg var r1}
\lim_{s\rightarrow 0}\frac{r_1(xs)}{r_1(s)}=x^{1/\ptt_0}.
\ee
This implies that $r_1$ is a regularly varying function at $0$, with $1/\ptt_0$ as the index of regular variation at $0$ (see \cite{HIZ98} for the definition). This allows us to use known estimates on regularly varying functions to tackle the integrals appearing in the following expression from \eqref{eq:Lmn}:
\begin{align}
\Vert\Le_{m_1,m_2}\Vert^2_{\mu_{\Om}}=
\left(\frac{\ltt m_1+m_2+1 \rtt !
    }{m_1!m_2!}\right)^2\int_0^1 r_1(s)^{2m_1}r_2(s)^{2m_2} ds\int_0^1 {r_1^{*}(s)}^{2m_1}{r_2^{*}(s)}^{2m_2}{ds}.
  \end{align}
Specifically, we will show that, as $j\rightarrow\infty$,
\bea
E_j&=&{\ltt{{m_2}_j}\rtt^{1+\frac{2m_0}{\ptt_0}}}{r_2(0)}^{-2{m_2}_j}{\tau\ltt\frac{1}{{m_2}_j}\rtt}^{-2m_0}\ltt\int_0^1 {r_1(s)}^{2m_0}{r_2(s)}^{2{m_2}_j}  ds\rtt\rightarrow C_{m_0},\label{eq:r1r2}
\\ 
E_j^*&=&{\ltt{{m_2}_j}\rtt^{1+\frac{2m_0}{\ptt_0^*}}}{r_2(0)}^{2{m_2}_j}{\tau\ltt\frac{1}{{m_2}_j}\rtt}^{2m_0}\ltt\int_0^{1}{r_1^*(s)}^{2m_0}{r_2^*(s)}^{2{m_2}_j}{ds}\rtt\rightarrow C^*_{m_0}, \label{eq:r1*r2*}
\eea
where $C_{m_0}, C^*_{m_0}>0$ are constants depending only on $m_0$, and $\ptt_0^*$ is the conjugate of $\ptt_0$.

Assume that \eqref{eq:r1r2} holds for any $\Om\in\mathcal R'$. Then it also holds for $\Om^*\in\mathcal R'$, i.e., 
    \[   \lim_{j\rightarrow\infty}{\ltt{{m_2}_j}\rtt^{1+\frac{2m_0}{\ptt_0^*}}}{{r_2^*(0)}^{-2{m_2}_j}{\tau^{*}\ltt\frac{1}{{m_2}_j}\rtt}^{-2m_0}}\ltt\int_0^1 {r_1^*(s)}^{2m_0}{r_2^*(s)}^{2{m_2}_j} ds\rtt = C'_{m_0}, 
    \]
for some positive constant $C'_{m_0}$. However, it follows from \eqref{eq:r_1^* r_2^* in s} that $r_2^*(0)=1/r_2(0)$ and $\tau^*\ltt\frac{1}{{m_2}_j}\rtt=\tau\ltt\frac{1}{{m_2}_j}\rtt^{-1}$. Thus, \eqref{eq:r1*r2*} holds if \eqref{eq:r1r2} holds. Further, by \eqref{eq:Lmn},
\beas
    \Vert \Le_{m_{0},m_{2j}}\Vert^2_{\mu_\Om}=\ltt{\frac{\ltt m_0+{m_2}_j+1 \rtt !
    }{m_0!{m_2}_j!}}\rtt ^{2} {m_2}_j^{-2(1+m_0)}E_jE_j^*.
\eeas
Thus, by \eqref{eq:r1r2}, \eqref{eq:r1*r2*}, and Stirling's approximation, we obtain that $\lim_{j\rightarrow\infty}\Vert \Le_{m_{0},m_{2j}}\Vert^2_{\mu_\Om}$ exists, and we are done in this case. 

It remains to prove \eqref{eq:r1r2}. For convenience, we denote ${m_2}_j$ by $m_j$. Assume $j$ is large enough so that $m_j>2$. Splitting the integral appearing in $E_j$ over $[0,1/2]$ and $(1/2,1]$, we write $E_j=E_j^1+E_j^2$. We first show that
\be\label{eq:est 0.5 to 1}
\lim_{j\rightarrow\infty}E_j^2=\lim_{j\rightarrow\infty}
{\ltt{m_j}\rtt^{1+\frac{2m_0}{\ptt_0}}}{{r_2(0)}^{-2m_j}{\tau\ltt\frac{1}{m_j}\rtt}^{-2m_0}}\ltt\int_\frac{1}{2}^{1}{r_1(s)}^{2m_0}{r_2(s)}^{2m_j}ds\rtt= 0.
\ee

Since $r_2(s)$ is a strictly decreasing function on $[0,1]$,
$
r_2(1)=0 <r_2(s)<r_2\ltt\frac{1}{2}\rtt
$, for any $s\in \ltt\frac{1}{2},1\rtt$. From the definition of $\tau$, 
\[
\tau\ltt\frac{1}{m_j}\rtt> b_1\exp\ltt-\int_\frac{1}{m_j}^{1}\frac{dt}{t}\rtt=\frac{b_1}{m_j}.
\]
Thus, there is a positive constant $D_{m_0}$ depending only on $m_0$ such that
\[
E_j^2\leq D_{m_0}{m_j^{1+2m_0+\frac{2m_0}{\ptt_0}}}{r_2(0)}^{-2m_j}{r_2\ltt\frac{1}{2}\rtt}^{2m_j},
\]
Since $r_2(\frac{1}{2})<r_2(0)$, the right-hand side tends to $0$ as $j\rightarrow\infty$.
Thus, \eqref{eq:est 0.5 to 1} holds.

The limit of $E_j^1$ is handled using the dominated convergence theorem, after applying an appropriate change of variables. Let $s=w/m_j$. Then, using \eqref{eq:r1}, we have that
\bes
E_j^1=\int_0^{\infty}  \mathbbm{1}_{[0,m_j/2]}
r_2\ltt\frac{w}{m_j}\rtt^{2m_j}{r_2(0)}^{-2m_j} w^\frac{2m_0}{\ptt_0} 
\tau\ltt\frac{w}{m_j}\rtt^{2m_0} \tau\ltt\frac{1}{m_j}\rtt^{-2m_0} dw,
\ees
where $\mathbbm{1}_A$ denotes the indicator function of $A\subset\R$. 
From $\eqref{eq:r_2 in s}$, it follows that $r_2$ is a $\cont^1$-smooth,  decreasing function on $(0,1)$. Since $\check{p}$ extends as a continuous function on $[0,1]$, $r_2^{\prime}$ extends continuously at $0$, with $\lim\limits_{s\rightarrow0}r_2^{\prime}(s)=-1/({r_2(0)\ptt_0})$. 
Thus, for each $w\in [0,\infty)$,
\begin{align}\label{eq:est2 r_2}
\lim_{j\rightarrow\infty}
{r_2\ltt\frac{w}{m_j}\rtt}^{2m_j}{r_2(0)}^{-2m_j}
&=\lim_{j\rightarrow\infty}\exp\left(m_j\ltt \log r_2\ltt\frac{w}{m_j}\rtt-\log r_2(0)\rtt\right)\notag \\
&=\exp\left(w\ltt\log {r_2}\rtt^{\prime}(0)\right)=\exp\ltt-{2w}{\ptt_0}^{-1}\rtt.
\end{align}
Also, from the expression of $\tau$ it follows that for each $x>0$, $\lim\limits_{s\rightarrow0}\frac{\tau(xs)}{\tau(s)}=1$. Hence,
\begin{align*}
\lim_{j\rightarrow\infty}r_2\ltt\frac{w}{m_j}\rtt^{2m_j}{r_2(0)}^{-2m_j} w^\frac{2m_0}{\ptt_0} 
\tau\ltt\frac{w}{m_j}\rtt^{2m_0} &\tau\ltt\frac{1}{m_j}\rtt^{-2m_0}=\exp\ltt-\frac{2}{\ptt_0}w\rtt w^{2m_0/\ptt_0}
\end{align*}
for each fixed $w\in[0,\infty)$. We now bound the integrand of $E_j^1$ by a function in $L^1(0,\infty).$

From the discussion about $r_2$ above, we have that ${w^{-1}}{\log\ltt{r_2(w)}{r_2(0)}^{-1}\rtt}$ is a negative continuous function on
$[0,\frac{1}{2}]$. Thus, there is a $\beta>0$ such that for each $m_j$,
\bea\label{eq:est1 r_2}
{r_2\ltt\frac{w}{m_j}\rtt}^{2m_j}{r_2(0)}^{-2m_j}\leq  \exp\ltt-\beta w\rtt\quad\forall w\in \left[0,\frac{m_j}{2}\right].
\eea
By \cite[Theorem 2]{HIZ98} on regularly varying functions, the convergence in \eqref{eq:reg var r1} is uniform on $x\in[0,1]$. In particular, the functions 
\be\label{eq:unif on [0,1]}
w^{1/{\ptt_0}} {\tau\ltt\frac{w}{m_j}\rtt}{\tau\ltt\frac{1}{m_j}\rtt}^{-1}\ \text{ are uniformly bounded on $[0,1]$.}
\ee
Next, assume $w\in(1,\frac{m_j}{2}]$.
From the definition of $\tau$, 
\be\label{eq:bd on rest}
{\tau\ltt\frac{w}{m_j}\rtt}{\tau\ltt\frac{1}{m_j}\rtt}^{-1}=\exp\int_\frac{1}{m_j}^\frac{w}{m_j}\ltt\frac{1}{\ptt_0}-\frac{1}{\ptt(t)}\rtt\frac{dt}{t}
<\exp\ltt\int_\frac{1}{m_j}^\frac{w}{m_j}\frac{1}{t}dt\rtt\,={w}.
\ee
Combining \eqref{eq:unif on [0,1]} and \eqref{eq:bd on rest} with (\ref{eq:est1 r_2}), we obtain a $D_{m_0}'>0$ such that 
\bes
\mathbbm{1}_{[0,m_j/2]}r_2\ltt\frac{w}{m_j}\rtt^{2m_j}{r_2(0)}^{-2m_j} w^{2m_0/\ptt_0} 
\tau\ltt\frac{w}{m_j}\rtt^{2m_0} \tau\ltt\frac{1}{m_j}\rtt^{-2m_0}\leq D'_{m_0}f(w),
\ees
where 
\bes
f(w)=\begin{cases}
 \exp\ltt-\beta w\rtt,&\ w\in[0,1],\\
 \exp\ltt-\beta w\rtt w^{2m_0\ltt\frac{1}{\ptt_0}+1\rtt},& w\in(1,\infty).
 \end{cases}
\ees
Finally, since $f\in L^1\ltt0,\infty\rtt$, by the dominated convergence theorem,
\begin{align}\label{eq:est 0 to 0.5}
\lim_{j\rightarrow\infty}E_j^1= \int_0^\infty \exp\ltt-\frac{2}{\ptt_0}w\rtt w^{2m_0/{\ptt_0}} dw.
\end{align}
Finally, combining \eqref{eq:est 0 to 0.5} and \eqref{eq:est 0.5 to 1}, we obtain \eqref{eq:r1r2}.

\qed

\subsection{Proof of Theorem~\ref{th:Leray} (ii)}\label{SS:proof1.4ii}
  The hypothesis on $\kappa_\Om$ implies that either the exponent $\check{p}$ or its conjugate $\check{p}^*$ is not bounded above. In other words, there exists a sequence $(x_l)_{l\in\N}\subset (0,\infty)$, such that 
  \be\label{eq:unbdd}
\lim_{l\rightarrow\infty}{\sqrt{\check{p}\ltt\frac{x_l}{1+x_l}\rtt\check{p}^*\ltt\frac{x_l}{1+x_l}\rtt}}=\infty.
\ee
For each $l\in\N$, let $({m}_{l,j},{n}_{l,j})_{j\in\N}\subset \N\times\N$ be a sequence such that $\lim_{j\rightarrow\infty}\min\{{m}_{l,j},{n}_{l,j}\}=\infty$ and $\lim_{j\rightarrow\infty}{m_{l,j}}/{n_{l,j}}=x_l$. We claim that for each $l\in\N$,
 \be\label{eq:counter conv}
\lim_{j\rightarrow\infty}\|\Le_{m_{l,j},n_{l,j}}\|^2_{\mu_\Om}= \frac{1}{2}\sqrt{\check{p}\ltt\frac{x_l}{1+x_l}\rtt\check{p}^*\ltt\frac{x_l}{1+x_l}\rtt}.
 \ee
Assuming \eqref{eq:counter conv}, we obtain from \eqref{eq:unbdd} that $\|\Le_{m,n}\|^2_{\mu_\Om}$ is not uniformly bounded for $(m,n)\in \N\times\N$. Thus, by Lemma~\ref{le:Leray bdd criteria rt} $(2)$, $\Le_b$ does not extend as a bounded operator on $L^2\ltt b\Om,\mu_\Om\rtt$.

The limit in  \eqref{eq:counter conv} can be established by making minor modifications to the proof of \cite[Theorem 45 (a)]{BaLa09}, which itself is a variation of Laplace's method for approximating integrals. Since the proof therein is rather long, we only provide a brief summary, indicating the modifications required in our case. 

Fix an $l\in\N$. For convenience, we denote $(m_{l,j},n_{l,j})$ by $(m_j,n_j)$. Now, note that 
\begin{align}\label{eq:neg main iden}
\frac{2\|\Le_{m,n}\|^2_{\mu_\Om}}{\sqrt{\check{p}\ltt{s_{m,n}}\rtt\check{p}^*\ltt {s_{m,n}}\rtt}}=\ltt\frac{\alpha_{m,n,0} A_{m,n,0}}{I_{m,n,0}}\rtt^2\cdot \frac{I_{m,n,1}}{\alpha_{m,n,1} A_{m,n,1}}\cdot \frac{I_{m,n,-1}}{\alpha_{m,n,-1}A_{m,n,-1}},
\end{align}
where $s_{m,n} =\frac{m}{m+n}$, $\alpha_{m,n,k}=g_{1,k}(s_{m,n})^{m}g_{2,k}(s_{m,n})^{n}$,
\[ 
 A_{m,n,k}=\sqrt{\frac{2m n}{\ltt\frac{2k}{\check p\ltt s_{m,n}\rtt}+1-k\rtt (m+n)^3}},\quad  I_{m,n,k}=\int_0^1 g_{1,k}(s)^{m}g_{2,k}(s)^{n} ds, \quad 
\]
$g_{1,k}(s)=r_1(s)^{2k}s^{1-k}$ and $g_{2,k}(s)=r_2(s)^{2k}(1-s)^{1-k}$, for $m,n\in\N$ and $k\in\{-1,0,1\}$. Thus, to prove \eqref{eq:counter conv}, it suffices to prove that 
\be\label{eq:toprove}
\lim_{j\rightarrow\infty}\frac{I_{m_j,n_j,k}}{\alpha_{m_j,n_j,k}A_{m_j,n_j,k}}=c
\ee
for some constant $c>0$. For a fixed $m,n,k$, one sets $s=s_{m,n}+tA_{m,n,k}$ to obtain  
\[
\frac{I_{m,n,k}}{A_{m,n,k}}=\alpha_{m,n,k}\int_{-\infty}^\infty \mathbbm 1_{\ltt-\frac{s_{m,n}}{A_{m,n,k}},\frac{1-s_{m,n}}{A_{m,n,k}}\rtt} g_{m,n,k}(t) dt,
\]
where $g_{m,n,k}(t)=\dfrac{g_{1,k}(s_{m,n}+tA_{m,n,k})^{m}g_{2,k}(s_{m,n}+tA_{m,n,k})^{n}}{\alpha_{m,n,k}}$. The rest of the proof relies on the fact that for all $k\in\{-1,0,1\}$,
\beas
\frac{2k}{\check{p}(s_{m_j,n_j})}+1-k&\leq& 2\\
C_{k,l     }=\inf_{j\in\N}\ltt\frac{2k}{\check{p}\ltt s_{{m_j},{n_j}}\rtt}+1-k\rtt&>&0.
\eeas
These are the necessary modifications of \cite[(5.4)]{BaLa09} and \cite[(5.5)]{BaLa09} needed in our case. 
As in the case of \cite[(5.13), (5.17), (5.19)]{BaLa09}, it follows that, for all $k\in\{-1,0,1\},$
\begin{enumerate}
    \item 
   $   \lim_{j\rightarrow\infty}{A_{m_j,n_j,k}}/{s_{m_j,n_j}}=\lim_{j\rightarrow\infty}{A_{m_j,n_j,k}}/{(1-s_{m_j,n_j})}=0
   $, 
\item  $\lim_{j\rightarrow\infty}g_{m_j,n_j,k}(t)=e^{-t^2}$ for all $t\in\mathbb R$, and
\item $g_{m_j,n_j,k}(t)\leq e^{C_{k,l}(1-2^{-|k|/2}|t|)}$ for all $t\in\R$, when $j$ is large enough.
\end{enumerate}
By the dominated convergence theorem, \eqref{eq:toprove} holds with $c=\sqrt \pi$. Combining this with \eqref{eq:neg main iden}, we obtain \eqref{eq:counter conv}.

\subsection{Remark}\label{se:rmk}
The case where $\kappa_3(\kappa_1\kappa_2)^{-1}$ is bounded away from zero on $b\Om_+$ but not uniformly continuous on $b\Om_+$ is quite subtle. This condition implies that the exponent $\check p:(0,1)\rightarrow (1,\infty)$ remains bounded away from $1$ and $\infty$, but fails to be uniformly continuous on $(0,1)$. Without loss of generality, we assume that the limit of $\check{p}(s)$ does not exist as $s\rightarrow 0$.  Under this assumption, the key difficulty in proving the $L^2(\mu_\Om)$-boundedness of $\Le_b$ lies in verifying conditions \eqref{eq:r1r2} and \eqref{eq:r1*r2*}. The crux of the issue is that $r_1$ and $r_1^*$ lose the property of being regularly varying at $0$ with a positive index. This follows from Karamata's representation theorem (see \cite[Theorem~1.2]{Se06}), which says that if $f:(0,1)\rightarrow\R$ is a regularly varying function at $0$, with index $\ell>0$, then
\[
f(s)= s^{\ell} \exp\ltt c(s)+\int_s^1 \frac{\epsilon(t)}{t}dt\rtt,\quad s\in(0,1),
\]
where $c,\epsilon$ are some bounded measurable functions with, $\lim_{s\rightarrow 0}c(s)=c,\,\lim_{s\rightarrow 0}\epsilon(s)=0$.
However, from \eqref{eq:r_1 in s} and \eqref{eq:r_1^* r_2^* in s}, it follows that such a representation would hold for $r_1$ or $r_1^*$ only if $\check{p}(s)$ has a finite limit as $s\rightarrow 0$, which contradicts our assumption. On the other hand, if we attempt to prove that $\Le_b$ does not extend as an $L^2(\mu_\Om)$-bounded operator by adapting the technique in Theorem~\ref{th:Leray}.ii, the attempt fails as \eqref{eq:counter conv} continues to hold. Since $\check{p}$ is bounded away from $1$, and $\infty$ on $(0,1)$, $\sqrt{\check{p}\ltt\frac{x_l}{1+x_l}\rtt\check{p}^*\ltt\frac{x_l}{1+x_l}\rtt}$ remains bounded above for all $x_l\in(0,\infty)$.

\section{Proof of Theorem~\ref{th:main thm 2}}\label{sub:proof of PW}

For the first claim, we use the series expansions obtained in Section \ref{subsec:series expansion} to write the $A^2\ltt\C^2,\nu_\Om\rtt$-norm of $\mathcal L(f)$ of any $f\in\mathcal{H}^2\ltt \Om,\mu_\Om\rtt$ in terms of the $\mathcal{H}^2\ltt \Om,\mu_\Om\rtt$-norm of $f$ and the operator norms of $\Le_{m_1,m_2}$, $m_1,m_2\in\N$. This allows us to use the characterization given by Lemma~\ref{le:Leray bdd criteria rt}. The bulk of the proof of the second claim is devoted to showing that the two measures $\nu_\Om$ and $\omega_\Om$ are comparable outside a compact set. For balls (in all dimensions), this is done in Section~\ref{sub:comp str cvx}. For other domains satisfying the hypothesis of the claim, this is done by comparing the Radon--Nikodym derivative of $\nu_\Om$ with respect to $\omega_\Om$ with that of the unit ball. 
\subsection{Proof of Theorem~\ref{th:main thm 2} $\mathbf{(i)}$}\label{subsec:proof PW part 1}
First, assume that $\Le_b$ extends as a bounded operator on $L^2\ltt b\Om,\mu_\Om \rtt$. Let $f\in \mathcal{H}^2\ltt \Om,\mu_\Om\rtt$. Then, by Lemma~\ref{pr:fourier series of Hardy space}, 
\bes
f(s,\theta_1,\theta_2)=\sum_{(m_1,m_2)\in \N^2}a_{m_1,m_2} {r_1(s)}^{m_1}{r_2(s)}^{m_2} e^{i({\theta_1 m_1+\theta_2 m_2})} \quad \left(\text{in }L^2(b\Om,\mu_\Om)\right).
\ees
with  $(a_{m_1,m_2})_{\N^2}$ satisfying \eqref{eq:amn condition}. By Lemma~\ref{le:series Bergman space}, $\mathcal L(f)\in A^2(\C^2,\nu_\Om)$ if 
\bes
\sum_{(m_1,m_2)\in \N^2}|t_{m_1,m_2}|^2(m_1+m_2+1)!^2 \left(\int_0^1 r_1^{*}(s)^{2m_1} r_2^{*}(s)^{2m_2}\,ds\right)<\infty,
\ees
where $t_{m_1,m_2}$ is related to $a_{m_1,m_2}$ as in \eqref{eq:coeff of power series}. From \eqref{eq:coeff of power series} and \eqref{eq:Lmn}, the uniform boundedness of $\Vert \Le_{m_1,m_2}\Vert_{\mu_\Om}^2$ from Lemma \ref{le:Leray bdd criteria rt}, and \eqref{eq:norm in hardy space}, we have that
\begin{align*}
& \sum_{(m_1,m_2)\in \N^2}|t_{m_1,m_2}|^2(m_1+m_2+1)!^2 \left(\int_0^1 r_1^{*}(s)^{2m_1} r_2^{*}(s)^{2m_2}\,ds\right)\\
\approx & \sum_{(m_1,m_2)\in \N^2} |a_{m_1,m_2}|^2 \|\Le_{m_1,m_2}\|_{\mu_\Om}^2 \left(\int_0^1 r_1(s)^{2m_1} r_2(s)^{2m_2} ds\right)\\
\approx &\sum_{(m_1,m_2)\in \N^2}|a_{m_1,m_2}|^2\left(\int_0^1 r_1(s)^{2m_1} r_2(s)^{2m_2} ds\right)\approx \|f\|^2_{\mu_\Om}<\infty. 
\end{align*}
Thus, $\mathcal{L}$ is an operator from $\mathcal H^2\ltt \Om,\mu
_\Om\rtt$ into $A^2\ltt\C^2,\nu_\Om\rtt$. Moreover, by Lemma~\ref{le:series Bergman space},
\bes
\Vert \mathcal L(f)\Vert^2_{\nu_\Om}
\approx
\sum_{(m_1,m_2)\in \N^2}|t_{m_1,m_2}|^2(m_1+m_2+1)!^2 \left(\int_0^1 r_1^{*}(s)^{2m_1} r_2^{*}(s)^{2m_2}\,ds\right)
\approx \|f\|^2_{\mu_\Om}<\infty.
\ees
Thus, $\mathcal L$ is injective. 

We now show that $\mathcal{L}$ is in fact onto $A^2\ltt\C^2,\nu_\Om\rtt$.
Let $F\in A^2\ltt\C^2,\nu_\Om\rtt$ be given by
\bes
F(z_1,z_2)=\sum_{\ltt m_1,m_2\rtt\in \N^2}\beta_{m_1,m_2}z_1^{m_1}z_2^{m_2}.
\ees  
By Lemma~\ref{le:series Bergman space},
\be\label{eq:main condition for onto}
\|F\|_{\mu_\Om}^2\approx \sum_{\ltt m_1,m_2\rtt\in\N^2}\vert\beta_{m_1,m_2}\vert ^2(m_1+m_2+1)!^2 \left(\int_0^1 (r_1^{*}(s))^{2m_1} (r_2^{*}(s))^{2m_2} ds\rtt.
\ee
Define $g$ on $b\Om$ as $
g(s,\theta_1,\theta_2)=\sum\limits_{(m_1,m_2)\in\N^2}\alpha_{m_1,m_2}r_1(s)^{m_1} r_2(s)^{m_2} e^{i(\theta_1m_1+\theta_2m_2)}$,
where
\be\label{eq:onto function}
\alpha_{m_1,m_2}={4}{\overline{\beta_{m_1,m_2}}}{\ltt m_1!m_2!\rtt}{\ltt\int_0^1{r_1(s)}^{2m_1} {r_2(s)}^{2m_2} ds\rtt^{-1}}.
\ee
Then, by \eqref{eq:Lmn} and the uniform boundedness of $\|\Le_{m_1,m_2}\|_{\mu_\Om}^2$, we have that
\begin{align*}
&\sum_{(m_1,m_2)\in\N^2}|\alpha_{m_1,m_2}|^2\left(\int_0^1r_1(s)^{2m_1}r_2(s)^{2m_2} ds\right)\\
\approx&\sum_{(m_1,m_2)\in\N^2}|\beta_{m_1,m_2}|^2(m_1+m_2+1)!^2\frac{1}{\|\Le_{m_1,m_2}\|_{\mu_\Om}^2} \left(\int_0^1 {r_1^*(s)}^{2m_1}{r_2^*(s)}^{2m_2} ds\right)\\
\approx&\sum_{(m_1,m_2)\in\N^2}|\beta_{m_1,m_2}|^2(m_1+m_2+1)!^2\left(\int_0^1 {r_1^*(s)}^{2m_1}{r_2^*(s)}^{2m_2} ds\right)<\infty.
\end{align*}
Thus, by Lemma~\ref{pr:fourier series of Hardy space}, $g\in \mathcal{H}^2\ltt \Om,\mu_\Om\rtt$, and
by Lemma~\ref{le:power series of laplace}, $\mathcal{L}(g)=F$. This concludes the proof of the claim that $\mathcal{L}$ is a normed space isomorphism between $\mathcal{H}^2\ltt \Om,\mu_\Om\rtt$ and $A^2\ltt\C^2,\nu_\Om\rtt$.

Now, we prove the converse claim by contraposition. Suppose $\Le_b$ does not extend as a bounded operator on $L^2\ltt b\Om,\mu_\Om\rtt$. By Lemma~\ref{le:Leray bdd criteria rt}, $\sup_{(m_1,m_2)\in \N^2}\|\Le_{m_1,m_2}\|_{\mu_\Om}^2=\infty$. Thus, there exists a sequence $\ltt m_{1_k},m_{2_k}\rtt_{k\in\Z_+}\subset\N\times\N$ such that $\|\Le_{m_{1_k},m_{2_k}}\|_{\mu_\Om}^2\geq k$, $\forall k\in \Z_+$. Consider the following function on $b\Om$ 
\bes
\widetilde{g}\ltt s,\theta_1,\theta_2\rtt=\sum_{(m_1,m_2)\in\N^2}\widetilde{\alpha}_{m_1,m_2}r_1(s)^{m_1} r_2(s)^{m_2} e^{i(\theta_1m_1+\theta_2m_2)},
\ees
where
\bes
\widetilde{\alpha}_{m_1,m_2}=\begin{cases}
\frac{1}{k}\left(\int_0^1{r_1(s)}^{2m_{1_{k}}}{r_2(s)}^{2m_{2_{k}}} ds\right)^{-\frac{1}{2}}, \quad & \text{when}\, m_1=m_{1_k},m_2=m_{2_k},\\
0, \quad &\text{otherwise}.
\end{cases}
\ees
Then, by Lemma~\ref{pr:fourier series of Hardy space}, $\widetilde{g}\in \mathcal H^2\ltt \Om,\mu_\Om\rtt$, but
\bes
\sum_{(m_1,m_2)\in \N^2}|\widetilde{\alpha}_{m_1,m_2}|^2\|\Le_{m_{1},m_{2}}\|_{\mu_\Om}^2\left(\int_0^1 (r_1(s))^{2m_1} (r_2(s))^{2m_2} ds\right)\geq \sum_{k=1}^\infty\frac{1}{k}=\infty.
\ees
Thus, by Lemma~\ref{le:series Bergman space}, $\mathcal{L}(\widetilde{g})\notin A^2\ltt\C^2,\nu_\Om\rtt$. This completes the proof of Theorem~\ref{th:main thm 2} $(i)$.
\subsection{Proof of Theorem \ref{th:main thm 2} $\mathbf{(ii)}$ assuming Lemma~\ref{le:comparision}} 
Let $\Om\in\rt$ be a domain that satisfies the hypothesis of Theorem \ref{th:main thm 2} (ii). Recall that $\Om^*$ is the dual complement of $\Om$. It suffices to prove that, for some $k>0$,
\begin{itemize}
    \item [(i)] for all $F\in\hol\ltt\C^2\rtt$,
    \be\label{eq:int est near 0}
    \int_{k\Om^*}\left\vert F(z)\right\vert^2 d\omega_\Om(z)\approx \int_{k\Om^*}\left\vert F(z)\right\vert^2 d\nu_\Om(z),
    \ee
\item [(ii)] for $z\in \C^2\setminus k\Om^*$,
\be\label{eq:est away from 0}
    \Vert e^{\langle z,\cdot\rangle}\Vert^{2}_{\mu_\Om}\approx e^{2H_\Om(z)} \|z\|^{-\frac{3}{2}}.
    \ee.
\end{itemize}

For the proof of (i), fix a $k>0$. Note that since both $e^{-2H_\Om(z)}$ and $\Vert e^{\langle z,\cdot\rangle}\Vert^{-2}_{\mu_\Om}$ are positive continuous functions on $\C^2$ and $\Om^*$ is a bounded set, it suffices to show that for all $F\in\hol(\C^2)$,
        \bes
        \int_{k\Om^*}\vert F(z)\vert^2 \Vert z\Vert^{\frac{3}{2}}\ltt dd^cH_\Om\rtt^2(z)\approx \int_{k\Om^*}\vert F(z)\vert^2\ltt dd^cH_\Om\rtt^2(z).
        \ees
         By the density of holomorphic polynomials in $\hol(\C^2)$, it further suffices to consider $F$ to be a holomorphic polynomial, i.e., of the form $F(z)=\sum_{m_1=0}^{k_1}\sum_{m_2=0}^{k_2}t_{m_1,m_2}z_1^{m_1}z_2^{m_2}.$ Recall, from the proof of Proposition~\ref{pr:cov} we have $T_\Om\ltt(0,k)\times b\Om\rtt=k\Om^*.$ Thus, applying \eqref{eq:cov} to $\psi(z)=\mathds 1_{k\Om^*}(z) |F(z)|^2 \|z\|^{3/2}$ and  $\psi(z)=\mathds 1_{k\Om^*}(z) |F(z)|^2$,  and mimicking the computation in Lemma \ref{le:series Bergman space}, we obtain that 
        \beas
         \int_{k\Om^*}\vert F(z)\vert^2 \| z\|^{\frac{3}{2}}\ltt dd^cH_\Om\rtt^2(z)
         &\approx& \sum_{m_1=0}^{k_1}\sum_{m_2=0}^{k_2}\frac{\left\vert t_{m_1,m_2}\right\vert^2}{\ltt2m_1+2m_2+\frac{7}{2}\rtt}\ltt\int_0^1 {r_1^*(s)}^{2m_1}{r_2^*(s)}^{2m_2} ds\rtt,
    \\
         \int_{k\Om^*}\vert F(z)\vert^2 \ltt dd^cH_\Om\rtt^2(z)
         &\approx& \sum_{m_1=0}^{k_1}\sum_{m_2=0}^{k_2}\frac{\left\vert t_{m_1,m_2}\right\vert^2}{2\ltt m_1+m_2+1\rtt}\ltt\int_0^1 {r_1^*(s)}^{2m_1}{r_2^*(s)}^{2m_2} ds\rtt.
        \eeas
    Finally, as $\ltt2m_1+2m_2+\frac{7}{2}\rtt\approx2\ltt m_1+m_2+1\rtt$, the proof of \eqref{eq:int est near 0} is complete.
    
       We now establish (ii) for a large enough $k>0$ (to be chosen later). We first reduce   \eqref{eq:est away from 0} to the integral estimate \eqref{eq:equiv form} below.  
Since the function ${e^{-2H_\Om(z)}}\|z\|^{\frac{3}{2}}\Vert e^{\langle z,\cdot\rangle}\Vert^{2}_{\mu_\Om}$ is continuous on $\C^2$, it suffices to prove \eqref{eq:est away from 0} for $z \in\C^2\setminus(\mathcal Z\cup k\Om^*)$.  For such a $z$, by Proposition \ref{pr:cov}, there exists $\ltt r,w \rtt\in [k,\infty)\times b\Om_+$ such that $z=T_\Om(r,w)=rT_\Om(1,w)$. Since $T_\Om(1,w)\in b\Om^*_+$, there exists $\ltt t,\phi_1,\phi_2\rtt\in(0,1)\times[0,2\pi)^2$ such that $z=r\ltt r_1^*(t)e^{i\phi_1},r_2^*(t)e^{i\phi_2}\rtt.$ Writing $ {e^{-2H_\Om(z)}}\|z\|^{\frac{3}{2}}\Vert e^{\langle z,\cdot\rangle}\Vert^{2}_{\mu_\Om}$ in terms of this parameterization, and using that $H_{\Om}\ltt T_\Om(r,w)\rtt=r$, we have that
    \bea      \frac{\|z\|^{\frac{3}{2}}}{{e^{2H_\Om(z)}}}\Vert e^{\langle z,\cdot\rangle}\Vert^{2}_{\mu_\Om}&=& \|z\|^{\frac{3}{2}}\int_{b\Om} e^{2\operatorname{Re}\langle\zeta,{z}\rangle-2H_\Om(z)} d\mu_{\Om}(\zeta)\notag\\
      &\approx& r^{\frac{3}{2}}\int_0^1\int_0^{2\pi}\int_0^{2\pi} e^{2r\ltt r_1(s)r_1^*(t)\cos(\theta_1-\phi_1)+r_2(s)r_2^*(t)\cos(\theta_2-\phi_2)-1\rtt} d\theta_1 d\theta_2 ds\notag \\
      &\approx&
      r^{\frac{3}{2}}\int_0^1\int_0^{\pi}\int_0^{\pi}  e^{2r\ltt r_1(s)r_1^*(t)\cos(\theta_1)+r_2(s)r_2^*(t)\cos(\theta_2)-1\rtt} d\theta_1 d\theta_2 ds.
    \label{eq:prelim estimate}
      \eea
Now, let $F_\Om:[0,1]^2\times[0,2\pi)^2\rightarrow \mathbb R$,
    \be\label{eq:FOm}
F_\Om(x,y,\psi_1,\psi_2)=  r_1(x)r_1^*(y)\cos(\psi_1)+r_2(x)r_2^*(y)\cos(\psi_2).
    \ee
    Since $\ltt r_1(x)e^{i\psi_1},r_2(x)e^{i\psi_2}\rtt\in b\Om$ and $\ltt r_1^*(y),r_2^*(y)\rtt\in b\Om^*$, we have that $F_\Om(x,y,\psi_1,\psi_2)\leq 1$ for all $\ltt x,y,\psi_1,\psi_2\rtt\in[0,1]^2\times[0,2\pi)^2 $.
    Moreover, since both ${\Om}$ and $\Om^*$ are strictly convex, for each $\zeta\in b\Om$, there exists a unique $\eta\in b\Om^*$ such that $\operatorname{Re}\langle \zeta,\eta\rangle=1$. Thus, $F_\Om(x,y,\psi_1,\psi_2)=1$ if and only if $x=y$ and $\psi_1=\psi_2=0$. Considering this, and the fact that $F_\Om$ is continuous on $[0,1]^2\times [0,2\pi)^2$, we obtain a $C>0$ such that  \begin{align}\label{eq:exp decay}
F_\Om(x,y,\psi_1,\psi_2)-1<-C,\quad \forall \ltt x,y,\psi_1,\psi_2\rtt\in[0,1]^2\times \left[\frac{\pi}{2},\pi\right]^2.
\end{align}
From \eqref{eq:prelim estimate} and \eqref{eq:exp decay}, it follows that proving \eqref{eq:est away from 0} is equivalent to proving 
\be\label{eq:equiv form}
   r^{\frac{3}{2}}\int_0^1\int_{0}^{\pi/2}\int_{0}^{\pi/2} e^{2r\ltt F_\Om(s,t,\theta_1,\theta_2)-1\rtt} d\theta_1 d\theta_2 ds\approx 1, \quad r>k. 
    \ee
    
To establish \eqref{eq:equiv form} for any $\Om$ satisfying the hypothesis of Theorem~\ref{th:main thm 2} $(ii)$, we use the following result which allows us to reduce the problem to the case of balls. 
\begin{lemma}\label{le:comparision} Let $\Om\in\rt$  satisfy the hypothesis of Theorem \ref{th:main thm 2} $(ii)$. Let $F_\Om$ be as in \eqref{eq:FOm}. There exist $C_1,C_2>0$ such that 
      \be\label{eq:comp w balls}
        C_1(1-F_{\mathbb B^2 })\leq 1-F_\Om\leq C_2(1-F_{\mathbb B^2 }),\quad \text{on }(0,1)^2\times \ltt 0,\frac{\pi}{2}\rtt^2.
       \ee
       \end{lemma}
       
Deferring the proof of Lemma~\ref{le:comparision} for now, we complete the proof of (ii). By Lemma~\ref{le:str cvx weight comp}, \eqref{eq:est away from 0}, and therefore \eqref{eq:equiv form},  holds for $\Om=\mathbb B^2 $ and $k=1$. In particular, for any $C>0$ 
\be\label{eq:ball est param}
   r^{\frac{3}{2}}\int_0^1\int_{0}^{\pi/2}\int_{0}^{\pi/2} e^{2rC\ltt F_{\mathbb B^2 }(s,t,\theta_1,\theta_2)-1\rtt} d\theta_1 d\theta_2 ds\approx 1,\quad r>\frac{1}{C}. 
    \ee
The estimate \eqref{eq:equiv form}, and therefore \eqref{eq:est away from 0}, now follows from \eqref{eq:comp w balls} and \eqref{eq:ball est param} for $k>1/C_1$. 
      

\subsection{Proof of Lemma \ref{le:comparision}}\label{ss:lemma comparison} Using elementary calculus, we optimize the functions $1-F_{\mathbb B^2 }-C_2(1-F_\Om)$ and $1-F_\Om-C_1(1-F_{\mathbb B^2 })$ 
on $(0,1)^2\times \ltt 0,{\pi}/{2}\rtt^2$, for appropriate choices of $C_1$ and $C_2$. It is simpler to do this for egg domains first, and then, for a general $\Om$.

\noindent {\em Upper bound for egg domains.} Let $p>1$ and $\Om_p=\{(z_1,z_2)\in\C^2:|z_1|^{p}+|z_2|^p<1\}$. We claim that, for all $(s,t,\theta_1,\theta_2)\in(0,1)^2\times \ltt 0,\pi/2\rtt^2$,
       \begin{align}{\label{eq:comp egg}}
      1-F_{\mathbb B^2 }(s,t,\theta_1,\theta_2)\leq C_{p}\ltt 1-F_{\Om_p}\ltt s,t,\theta_1,\theta_2\rtt\rtt,
      \end{align}
        where $C_p=
         p/2$, when $p\geq 2$ and
          $C_p=p/(2p-2)$, when
        $1<p<2.$

        Assume $p\geq 2$. Fix $(t,\theta_1,\theta_2)\in (0,1)\times \ltt0,{\pi}/{2}\rtt^2$ and consider $G_1:(0,1)\rightarrow \mathbb R$ given by
        \beas 
        G_1(s)&=& 1-F_{\mathbb B^2 }(s,t,\theta_1,\theta_2)-\frac{p}{2}\ltt 1-F_{\Om_p}\ltt s,t,\theta_1,\theta_2\rtt\rtt\\
        &=& 1-\frac{p}{2}-\ltt s^{\frac{1}{2}}t^{\frac{1}{2}}-\frac{p}{2}s^{\frac{1}{p}}t^{\frac{p-1}{p}}\rtt\cos\theta_1-\ltt(1-s)^{\frac{1}{2}}(1-t)^{\frac{1}{2}}-\frac{p}{2}(1-s)^{\frac{1}{p}}(1-t)^{\frac{p-1}{p}}\rtt\cos\theta_2,
        \eeas
        for $s\in(0,1)$. Then, for all $s\in(0,1)$,
        \begin{align*}
        G_1'(s)
        =\frac{1}{2}\ltt\ltt\frac{t}{s}\rtt^{\frac{p-1}{p}}-\ltt\frac{t}{s}\rtt^{\frac{1}{2}}\rtt \cos\theta_1+\frac{1}{2}\ltt\ltt\frac{1-t}{1-s}\rtt^{\frac{1}{2}}-\ltt\frac{1-t}{1-s}\rtt^{\frac{p-1}{p}}\rtt \cos\theta_2.
        \end{align*}
     Now using the fact that $ a^{\frac{p-1}{p}}> a^{\frac{1}{2}}$ whenever $a>1$, 
     we have that $G'(s)>0$ if $s\in(0,t)$ and $G'(s)<0$ if $s\in(t,1)$.
 Thus, $G_1$ has a global maximum on $(0,1)$ at $s=t$.
    However,
    \[
    G_1(t)=\ltt1-\frac{p}{2}\rtt\ltt1-t\cos \theta_1 -(1-t)\cos\theta_2\rtt\leq 0,
    \]
which proves the claim in the case when $p\geq 2$.

 The proof of \eqref{eq:comp egg} when $1<p\leq 2$ follows once we observe that $p^*\geq 2$, $F_{\Om_p}(s,t,\theta_1,\theta_2)=F_{\Om_p^*}(t,s,\theta_1,\theta_2)$ and $p^*/{2}=p/(2p-2)$.
   
\noindent {\em Lower bound for egg domains.} We claim that for all $(s,t,\theta_1,\theta_2)\in(0,1)^2\times(0,{\pi}/{2})^2$,
    \begin{align*}
         1-F_{\Om_p}\ltt s,t,\theta_1,\theta_2)\rtt\leq \frac{2}{(2C_{p})^*}\ltt1-F_{\mathbb B^2}(s,t,\theta_1,\theta_2)\rtt.
    \end{align*}
The same method of proof as in the previous case is applied to the function 
    \[
G_2(t)= 1-F_{\Om_p}(s,t,\theta_1,\theta_2)-\frac{2}{(2C_{p})^*}\ltt 1-F_{\mathbb B^2}\ltt s,t,\theta_1,\theta_2\rtt\rtt,\quad t\in(0,1),
\]
for a fixed  $(s,\theta_1,\theta_2)\in(0,1)\times (0,{\pi}/{2})^2$.

\noindent {\em Upper bound in the general case.} Let $\Om\in\rt$ satisfy the hypothesis of Theorem~\ref{th:main thm 2} $(ii)$. Then,  $p_{\ell}=\inf\{\check{p}(u):{u\in(0,1)}\}>1$ and $p_g=\sup\{\check{p}(u):{u\in(0,1)}\}<\infty$. We claim that, for all $\ltt s,t,\theta_1,\theta_2\rtt\in\ltt 0,1\rtt^2\times\ltt0,{\pi}/{2}\rtt^2$,
\begin{align}\label{eq:comp with egg 1}
        1-F_{\Om_{p_{\ell}}}\ltt s,t,\theta_1,\theta_2\rtt\leq \frac{p_g}{p_{\ell}}\ltt1-F_\Om\ltt s,t,\theta_1,\theta_2\rtt\rtt.
       \end{align}
      
From \eqref{eq:r_1 in s} and \eqref{eq:r_2 in s}, it follows that that 
\bea{\label{eq:r_1 with h_1}}
r_1(s)&=&s^{\frac{1}{p_{\ell}}}h_1(s),\,\text{where}\, h_1(s)=b_1\exp\ltt\int_s^1\ltt\frac{1}{p_{\ell}}-\frac{1}{\check{p}(u)}\rtt \frac{du}{u}\rtt,\\\label{eq:r_2 with h_2}
r_2(s)&=&(1-s)^{\frac{1}{p_{\ell}}}h_2(s),\,\text{where}\, h_2(s)=b_2\exp\ltt\int_0^s\ltt\frac{1}{p_{\ell}}-\frac{1}{\check{p}(u)}\rtt \frac{du}{1-u}\rtt.
\eea
Note that $h_1$ is a decreasing function and $h_2$ is an increasing function in $s$. 

 Now, fix $s,t\in(0,1)$ and consider the function $G:\left[0,{\pi }/{2}\right]^2\rightarrow \R$ given by
    \bea    G(\theta_1,\theta_2)&=& 1-F_{\Om_{p_{\ell}}}\ltt s,t,\theta_1,\theta_2\rtt-\frac{p_g}{p_{\ell}}\ltt1-F_\Om\ltt s,t,\theta_1,\theta_2\rtt\rtt\notag \\
    &=& 1-\frac{p_g}{p_{\ell}}-s^{\frac{1}{p_{\ell}}}t^{\frac{p_{\ell}-1}{p_{\ell}}}
    \ltt 1-\frac{p_g}{p_{\ell}}\frac{h_1(s)}{h_1(t)}\rtt\cos\theta_1
    \\ &&\qquad \qquad -(1-s)^{\frac{1}{p_{\ell}}}(1-t)^{\frac{p_{\ell}-1}{p_{\ell}}}\ltt1-\frac{p_g}{p_{\ell}}\frac{h_2(s)}{h_2(t)}\rtt\cos\theta_2,\label{eq:G}
    \eea
    for $(\theta_1,\theta_2)\in(0,\pi/2)^2.$
    Since $G$ is continuous on the rectangle $\left[0,{\pi}/{2}\right]^2$, $G$ attains its maximum on $\left[0,{\pi}/{2}\right]^2$. In fact, we claim that the maximum is attained at one of the corners. Suppose the maximum is attained at a point $\ltt\theta_1^{'},\theta_2^{'}\rtt\in \ltt0,{\pi}/{2}\rtt^2$. Then $0= 
        \nabla G\ltt\theta_1^{'},\theta_2^{'}\rtt$, which, is only possible if $0= 1-\frac{p_g}{p_{\ell}}\frac{h_1(s)}{h_1(t)}=1-\frac{p_g}{p_{\ell}}\frac{h_2(s)}{h_2(t)}$, since $\sin(\theta_j^{'})\neq 0$ for $j=\{1,2\}$. In this case, $\nabla G\equiv 0$, and $G$ is a constant function. 
    
Alternatively, suppose the maximum is attained at $\ltt\theta_1^{'},0\rtt$ for some $\theta_1^{'}\in \ltt0,\pi/2\rtt$. Then, 
\[
0=
\partl{G}{\theta_1}\ltt\theta_1',0\rtt
=-s^{\frac{1}{p_{\ell}}}t^{\frac{p_{\ell}-1}{p_{\ell}}}\ltt1-\frac{p_g}{p_{\ell}}\frac{h_1(s)}{h_1(t)}\rtt\sin\theta_1'.
\]
This implies that $\smpartl{G}{\theta_1}\ltt\cdot,0\rtt\equiv 0$, and $G(\cdot,0)$ is a constant function. The interior of the other three edges can be tackled in a similar fashion. Thus, to show that $G\leq 0$ on $[0,\pi/2]^2$, we show that the value of $G$ is nonpositive at the four corners. 
\begin{itemize}
    \item [(1)]  Since $p_{\ell}\leq p_g$, $G\ltt\frac{\pi}{2},\frac{\pi}{2}\rtt=1-p_g/{p_{\ell}}\leq 0$.

    \item [(2)] Note that
    \begin{align*}    G\ltt\frac{\pi}{2},0\rtt=1-\frac{p_g}{p_{\ell}}-(1-s)^\frac{1}{p_{\ell}}(1-t)^\frac{p_{\ell}-1}{p_{\ell}}\ltt1-\frac{p_g}{p_{\ell}}\frac{h_2(s)}{h_2(t)}\rtt.
    \end{align*}
For a fixed $t\in(0,1)$, we consider the right-hand side to be a function of $s$, i.e., let
\[
L(s)=1-\frac{p_g}{p_{\ell}}-(1-s)^\frac{1}{p_{\ell}}(1-t)^\frac{p_{\ell}-1}{p_{\ell}}\ltt1-\frac{p_g}{p_{\ell}}\frac{h_2(s)}{h_2(t)}\rtt. 
\]
As $r_1$ and $r_2$ extend continuously to $[0,1]$, so does $L$. We show that $L$ is nonpositive at $s=0,1$ and at any $s\in(0,1)$ such that $L'(s)=0$. Thus, $L(s)\leq 0$ on $[0,1]$.
\begin{itemize}
    \item [(a)] Since $p_{\ell}\leq p_g$, $h_2$ is an increasing function and $h_2(0)=b_2$,
\begin{align*}
    L(0)
    =\ltt1-\frac{p_g}{p_{\ell}}\rtt+(1-t)^\frac{p_{\ell}-1}{p_{\ell}}\ltt\frac{p_g}{p_{\ell}}\frac{b_2}{h_2(t)}-1\rtt\leq 0. 
\end{align*}

\item [(b)] Since $p_{\ell}\leq p_g$, $L(1)=1-{p_g}/{p_{\ell}}\leq 0$. 

\item [(c)] Suppose $L'(s_t)=0$ at some
$s_t\in(0,1)$. Then, since $h_2'(s)=\frac{h_2(s)}{1-s}\ltt\frac{1}{p_{\ell}}-\frac{1}{\check p(s)}\rtt$, 
     \begin{align*}
    0=L{'}(s_t)
    =\ltt\frac{1-t}{1-s_t}\rtt^{\frac{p_{\ell}-1}{p_{\ell}}}\ltt\frac{1}{p_{\ell}}-\frac{p_g}{p_{\ell}}\frac{h_2(s_t)}{\check{p}(s_t)h_2(t)}\rtt.
     \end{align*}
     Thus, $\check{p}(s_t)=p_g{h_2(s_t)}/{h_2(t)}$. Plugging this into $L(s_t)$, we have that
    \bes
    L(s_t)
    =\ltt1-\frac{p_g}{p_{\ell}}\rtt+(1-s_t)^{\frac{1}{p_{\ell}}}(1-t)^{\frac{p_{\ell}-1}{p_{\ell}}}\ltt\frac{\check{p}(s_t)}{p_{\ell}}-1\rtt
    \leq 0.
    \ees
    From (a), (b) and (c), we obtain that $G\ltt\frac{\pi}{2},0\rtt=L(s)\leq 0$.
    \end{itemize}

    \item [(3)] The same method of proof as in (2) works for $G(0,\pi/2)$ to yield that  $G\ltt0,{\pi}/{2}\rtt\leq 0$.

    \item [(4)] Note that $G(0,0)$ is 
        \beas
        1-\frac{p_g}{p_{\ell}}
        -s^{\frac{1}{p_{\ell}}}t^{\frac{p_{\ell}-1}{p_{\ell}}}\ltt1-\frac{p_g}{p_{\ell}}\frac{h_1(s)}{h_1(t)}\rtt
        -(1-s)^{\frac{1}{p_{\ell}}}(1-t)^{\frac{p_{\ell}-1}{p_{\ell}}}\ltt1-\frac{p_g}{p_{\ell}}
    \frac{h_2(s)}{h_2(t)}\rtt.
    \eeas
    For a fixed $t\in(0,1)$, let $L(s)$ denote the above expression. Then, 
\begin{align*}
L^{'}(s)
&=\ltt\frac{t}{s}\rtt^{\frac{p_{\ell}-1}{p_{\ell}}}\ltt\frac{p_g}{p_{\ell}\check{p}(s)}\frac{h_1(s)}{h_1(t)}-\frac{1}{p_{\ell}}\rtt+\ltt\frac{1-t}{1-s}\rtt^{\frac{p_{\ell}-1}{p_{\ell}}}\ltt\frac{1}{p_{\ell}}-\frac{p_g}{p_{\ell}\check{p}(s)}\frac{h_2(s)}{h_2(t)}\rtt.
\end{align*} 
Now, using that $h_1$ is increasing, $h_2$ is decreasing, $p_g\geq\check p(s)$ for all $s\in(0,1)$, and $p_{\ell}-1>0$, we obtain that $
{L}^{'}(s)\geq 0$ for all $s\in(0,t]$. Similarly, ${L}^{'}(s)\leq 0$ for all $s\in [t,1)$. Thus, $L$ attains its maximum at $s=t$, which is $L(t)=0$. Consequently, $G(0,0)\leq 0$.
\end{itemize}
   
Since $G$ must attain it maximum at one of the corners, we obtain \eqref{eq:comp with egg 1} from (1)-(4).

\noindent{\em Lower bound for a general $\Om$.} Using the same method of proof as in the case of \eqref{eq:comp with egg 1}, but switching the roles of $s$ and $t$, we obtain that for any $\ltt s,t,\theta_1,\theta_2\rtt\in(0,1)^2\times \left[0,{\pi}/{2}\right]^2$,
      \begin{align}\label{eq:comp wth egg 2}
      1-F_\Om\ltt s,t,\theta_1,\theta_2\rtt\leq \frac{p^*_g}{p^*_\ell}\ltt1-F_{\Om_{p_{\ell}}}\ltt s,t,\theta_1,\theta_2\rtt\rtt,
       \end{align}
    where $p^*_g=\sup\{\check{p}^*(u):u\in(0,1)\}$, $p^*_\ell=\inf\{\check{p}^*(u):u\in(0,1)\}$.

\section{Proofs of Theorem~\ref{th:main} and Theorem~\ref{th:negative}}\label{sub:counter}

The proofs of Theorem~\ref{th:main} and \ref{th:negative} are now relatively straightforward. 

\subsection{Proof of Theorem~\ref{th:main}} 
By \eqref{eq:p in phi} and \eqref{eq:reln between kappa an p}, we obtain that for $\zeta=\ltt \zeta_1,\zeta_2\rtt\in b\Om_+$,
\[
\kappa_\Om(\zeta)=\frac{\left|\zeta_1\right|\phi''(\left|\zeta_1\right|)\phi(\left|\zeta_1\right|)}{\phi'(\left|\zeta_1\right|)\ltt1+\phi'(\left|\zeta_1\right|)^2\rtt^\frac{1}{2}}.
\]
The hypothesis on $\kappa_\Om$ implies that $\lim_{|\zeta_1|\rightarrow 0}\kappa_\Om(\zeta)$ and $\lim_{|\zeta_1|\rightarrow b_2}\kappa_\Om(\zeta)$ are finite and positive. Due to the strong convexity of $b\Om_+$, $\phi'$ is a strictly decreasing negative 
function on $(0,b_1)$. Hence, $\lim_{\left|\zeta_1\right|\rightarrow 0}\phi'(\left|\zeta_1\right|)=l\in(-\infty,0]$. However, if $l\in(-\infty,0)$, then from the above expression it follows that $\lim_{\left|\zeta_1\right|\rightarrow 0} \left|\zeta_1\right|\phi''(\left|\zeta_1\right|)<0.$ This implies that $\lim_{\left|\zeta_1\right|\rightarrow 0}\phi'(\left|\zeta_1\right|)=\infty$, which is a contradiction. Thus, $\lim_{\left|\zeta_1\right|\rightarrow 0}\phi'(\left|\zeta_1\right|)=0$. Similarly, we can show that $\lim_{\left|\zeta_1\right|\rightarrow b_1}\phi'(\left|\zeta_1\right|)=-\infty$. As discussed in Section~\ref{sub:domains}, this implies that $\Om$ has a $\cont^1$-smooth boundary, and thus, $\Om\in\rt$. Furthermore, since $\kappa_\Om$ is bounded away from zero on $b\Om_+$ and extends continuously to the axes, it follows that $\Om\in\mathcal{R}'$. Theorem~\ref{th:main} now follows from Theorem~\ref{th:Leray} and Theorem~\ref{th:main thm 2}.

\subsection{Proof of Theorem~\ref{th:negative}} Note that 
    \bes
        \Om=\{(z_1,z_2)\in\C^2:|z_1|+|z_2|<1\}
    \ees
is a convex Reinhardt domain that is $\cont^2$-smooth but not strongly convex away from the axes. Thus, the preliminaries established in Sections~\ref{sub:prelim} and~\ref{sub:tech tools} do not directly apply to $\Om$. However, several of the features of domains in $\rt$ can be recovered for $\Om$ using the same arguments. We note the relevant features here, but leave the details to the reader.
\begin{itemize}
    \item [(i)] The boundary of $\Om$ admits the
  following parameterization
  \bea 
    \vartheta:[0,1]\times [0,2\pi)^2&\rightarrow& b\Om,\notag \\
      (s,\theta_1,\theta_2)&\mapsto&\ltt se^{i\theta_1},(1-s)e^{i\theta_2}\rtt.
      \label{eq:param of diamond}
        \eea
    \item [(ii)] The Minskwoski functional and the support function of $\Om$ are given by $m_\Om(z_1,z_2)=|z_1|+|z_2|$ and $H_\Om(z_1,z_2)=\max\{|z_1|,|z_2|\}$, respectively.  
    \item [(iii)] The boundary Monge--Amp{\`e}re measure of $\Om$ with respect to $m_\Om$ (in terms of the given parametrization) is
      \[
   \left(\vartheta^*\mu_\Om\right)(s,\theta_1,\theta_2)=\frac{1}{16\pi^2}\:ds\,d\theta_1 d\theta_2. 
        \]
    \item [(iv)] The Monge--Amp{\`e}re measure of $H_\Om$ is a Radon measure on $\C^2$ that is supported on the real hypersurface $M=\{(z_1,z_2)\in\C^2:|z_1|=|z_2|\}$; see  \cite[Theorem 1]{BS08}. In fact, following the computations in \cite[Theorem 1]{BS08} for $E=\C^2\setminus \mathcal{Z}$ and $u=H_\Om$, we have that for $z=(z_1,z_2)\in\C^2$,
        \begin{align*}            
        (dd^cH_\Om)^2(z)&\approx \mathds{1}_M(z)\frac{1}{|z_1||z_2|} \sigma_M(z)\label{eq:MA for max}\\
        &\approx j^*\ltt\frac{\overline{z_1}}{|z_1|}\;dz_1 dz_2 d\overline{z_2}-\frac{\overline{z_2}}{|z_2|}\;dz_1 dz_2 d\overline{z_1}\rtt. \notag
         \end{align*}
         Parametrizing $M$ via $\vartheta_M
         \ltt r,\psi_1,\psi_2\rtt=\ltt re^{i\psi_1},re^{i\psi_2}\rtt$, we obtain that
        \be\label{eq:MA bi disc}
       \vartheta_M^* (dd^cH_\Om)^2\ltt r,\psi_1,\psi_2\rtt\approx dr d\psi_1 d\psi_2,\quad (r,\psi_1,\psi_2)\in(0,\infty)\times[0,2\pi)^2.
        \ee
    \item [(v)] Analogous to  Lemma~\ref{pr:fourier series of Hardy space}, any $f\in\mathcal{H}^2\ltt \Om,\mu_\Om\rtt$ if and only if     
        \be\label{eq:fourier of f in L1 ball}
f(s,\theta_1,\theta_2)=\sum_{(m_1,m_2)\in \N^2}a_{m_1,m_2} {s}^{m_1}{(1-s)}^{m_2} e^{i({\theta_1 {m_1}+\theta_2 {m_2}})} \quad \left(\text{in }L^2(b\Om,\mu_\Om)\right)
\ee
for some sequence $(a_{m_1,m_2})_{\N^2}\subset\C$
satisfying 
\begin{align}\label{eq:amn condition in L1 ball}
\sum_{\N^2}|a_{m_1,m_2}|^2\left(\int_0^1 s^{2m_1} (1-s)^{2m_2}ds\right)=\sum_{ \N^2}|a_{m_1,m_2}|^2 \dfrac{2m_1!2m_2!}{(2m_1+2m_2+1)!}<\infty.
\end{align}
  Moreover,
\be\label{eq:norm in L1 ball}
\|f\|^2_{\mu_\Om}\approx\sum_{(m_1,m_2)\in \N^2}|a_{m_1,m_2}|^2\dfrac{2m_1!2m_2!}{(2m_1+2m_2+1)!}.
\ee   
\item [(vi)] Analogous to Lemma \ref{le:power series of laplace}, if $f$ is as in \eqref{eq:fourier of f in L1 ball}, we have that $\mathcal L(f)$ is entire on $\C^2$ and admits the following power series expansion:
\be\label{eq:F-L series for L1 ball}
\mathcal{L}(f)(z_1,z_2)=\frac{1}{4}\sum_{(m_1,m_2)\in \N^2}\dfrac{\overline{a_{m_1,m_2}}}{m_1!m_2!}\dfrac{2m_1!2m_2!}{(2m_1+2m_2+1)!}z_1^{m_1}z_2^{m_2},
\ee
where the above series converges uniformly on compact subsets of $\C^2$.
\item [(vii)] Analogous to Lemma~\ref{le:series Bergman space}, if $F\in\hol(\C^2)$ has expansion $\sum_{\N^2}\beta_{m_1,m_2}z_1^{m_1}z_2^{m_2}$, then
\begin{align}
\Vert F\Vert^2_{\nu_\Om}
&=\int_0^\infty\int_0^{2\pi}\int_0^{2\pi}\left\vert F\ltt re^{i\psi_1},re^{i\psi_2}\rtt\right\vert^2 e^{-2r} r^{\frac{3}{2}}
d\psi_1 d\psi_2 dr\notag \\
&\approx\sum\limits_{(m_1,m_2)\in \N^2}\frac{|\beta_{m_1,m_2}|^2}{2^{2m_1+2m_2}} \Gamma\ltt 2m_1+2m_2+5/2\rtt.
\label{eq:Berg condition lindholm max}
\end{align}

\end{itemize}

    It remains to characterize the power series expansions of entire functions in $A^2(\Om,\omega_\Om)$. For this, we analyze the function $\|e^{\langle z,.\rangle}\|^2_{\mu_\Om}$ on $M$.
\begin{lemma}\label{le:weight eqiv L1 ball}
Let $\Om, \mu_\Om$ be as above. Then
\be\label{eq:weight eqiv L1}
\|e^{\langle z,.\rangle}\|^2_{\mu_\Om}\approx {e^{2H_\Om(z)}}{\|z\|^{-1}}, \quad z\in M\cap\{z\in\C^2:H_\Om(z)>1\}.
\ee
\end{lemma}
\begin{proof}
Parametrizing $M\cap\{z\in\C^2:H_\Om(z)>1\}$ as $\{\ltt re^{i\psi_1},re^{i\psi_2}\rtt:r>0, \psi_1,\psi_2\in[0,2\pi)\}$, we have that
    \begin{align*}
       {e^{-2H_\Om(z)}}\|z\| \|e^{\langle z,.\rangle}\|^2_{\mu_\Om}&\approx r\int_0^1\int_0^{2\pi}\int_0^{2\pi} e^{2r\ltt s\cos(\theta_1-\psi_1)+(1-s)\cos(\theta_2-\psi_2)-1\rtt} d\theta_1 d\theta_2 ds\\
    &\approx r\int_0^1\int_0^{\pi}\int_0^{\pi} e^{2r\ltt s\cos(\theta_1)+(1-s)\cos(\theta_2)-1\rtt} d\theta_1 d\theta_2 ds\\
    &=r\int_0^1\int_0^{\pi}\int_0^{\pi} e^{-4r\ltt s\sin^2(\theta_1/2)+(1-s)\sin^2(\theta_2/2)\rtt} d\theta_1 d\theta_2 ds.
    \end{align*}
Now, using that, $\sin^2\frac{\theta}{2}\approx \theta^2$ on $[0,\pi]$, and that
\[
r\int_0^1 \int_0^{\pi}\int_0^{\pi}e^{-2rC\ltt s\theta_1^2+(1-s)\theta_2^2\rtt}d\theta_1 d\theta_2 ds\approx\frac{1}{C}\int_0^1 \frac{1}{s^{\frac{1}{2}}(1-s)^{\frac{1}{2}}}ds,
\]
we obtain the desired result. 
\end{proof}

Combining the above lemma with \eqref{eq:MA bi disc}, we conclude that if $F\in\hol(\C^2)$ has expansion $\sum_{\N^2}\beta_{m_1,m_2}z_1^{m_1}z_2^{m_2}$, then 
\bea\label{eq:Berg condition max}
\Vert F\Vert^2_{\omega_\Om}&\approx&\int_0^\infty\int_0^{2\pi}\int_0^{2\pi}\left\vert F\ltt re^{i\psi_1},re^{i\psi_2}\rtt\right\vert^2 e^{-2r} r
d\psi_1 d\psi_2 dr\\
&\approx&\sum\limits_{(m_1,m_2)\in \N^2}\frac{|\beta_{m_1,m_2}|^2}{2^{2m_1+2m_2}} \ltt 2m_1+2m_2+1\rtt!.\notag
\eea 

To see that $A^2\ltt\C^2,\nu_\Om\rtt\subsetneq A^2\ltt\C^2,\omega_\Om\rtt,$ consider $G(z_1,z_2)=t_{m_1,m_2}z_1^{m_1}z_2^{m_2}$, where
\begin{align*}
      t_{m_1,m_2}=\begin{cases}
          \frac{2^{2k}}{k^{\frac{3}{4}}{\ltt4k+1\rtt!}^\frac{1}{2}},& \text{when } m_1=m_2=k>1,\\
          0,&\text{otherwise}.
      \end{cases}  
    \end{align*}
Then $G\in\hol\ltt\C^2\rtt$, and by \eqref{eq:Berg condition max}, $\|G\|^2_{\omega_\Om} \approx \sum\limits_{k=1}^\infty \frac{1}{k^{\frac{3}{2}}}<\infty$. Thus $G\in A^2\ltt\C^2,\omega_\Om\rtt$. However, using Stirling's approximation and \eqref{eq:Berg condition lindholm max},  $\|G\|^2_{\nu_\Om}\approx\sum\limits_{k=1}^\infty \frac{1}{k}=\infty$. Thus $G\notin A^2\ltt\C^2,\nu_\Om\rtt$.

To see that $\mathcal L(\mathcal H^2(\Om,\mu_\Om))\subsetneq A^2(\C^2,\nu_\Om)$, consider $F(z_1,z_2)=\sum\limits_{m_1,m_2=0}^\infty b_{m_1,m_2} z_1^{m_1} z_2^{m_2},$ where
    \begin{align*}
      b_{m_1,m_2}=\begin{cases}
          \frac{2^{2k}}{k^{\frac{3}{4}}{\Gamma\ltt4k+\frac{5}{2}\rtt}^\frac{1}{2}},& \text{when } m_1=m_2=k>1,\\
          0,&\text{otherwise}.
      \end{cases}  
    \end{align*}
Then, $F\in\hol\ltt\C^2\rtt$, and, by \eqref{eq:Berg condition lindholm max}
$
\Vert F\Vert^2_{\nu_\Om}\approx \sum\limits_{k=1}^\infty \frac{1}{k^{\frac{3}{2}}}<\infty.
$
Thus, $F\in A^2\ltt\C^2,\nu_\Om\rtt$. Now, suppose there is an $f\in \mathcal{H}^2\ltt \Om,\mu_\Om\rtt$ such that $\mathcal{L}(f)=F$. Then, $f$ admits an expansion as in \eqref{eq:fourier of f in L1 ball} for some sequence $(a_{m_1,m_2})_{\N^2}$, satisfying \eqref{eq:amn condition in L1 ball}. Then, by \eqref{eq:F-L series for L1 ball},
\begin{align*}
    a_{m_1,m_2}&=4\frac{ \overline{b_{m_1,m_2}}}{2m_1!2m_2!}{\ltt2m_1+2m_2+1\rtt!}{\ltt m_1!m_2!\rtt},\\
    &=\begin{cases}
          \frac{2^{2k+2} k!^2}{k^{\frac{3}{4}}(2k!)^2}{\Gamma\ltt4k+\frac{5}{2}\rtt}^{-\frac{1}{2}}(4k+1)!,& \text{when } m_1=m_2=k>1,\\
          0,&\text{otherwise}.
      \end{cases} 
\end{align*}
Thus, by \eqref{eq:amn condition in L1 ball},
\beas
\Vert  f\Vert_{\mu_\Om}^2 \approx
\sum_{k=1}^\infty \frac{2^{4k}(k!)^4}{k^\frac{3}{2}(2k!)^2}\frac{(4k+1)!}{\Gamma\ltt4k+\frac{5}{2}\rtt}=\infty,
\eeas
where we have used the uniform boundedness of 
\[
\frac{k^{\frac{1}{2}}(2k!)^2}{2^{4k}(k!)^4}{\frac{\Gamma\ltt 4k+\frac{5}{2}\rtt}{(4k+1)!}}
\]
by Stirling's approximation. This is a contradiction, and our proof is complete. 

\section{Comparison of two weighted Bergman spaces}\label{sub:comp str cvx} For a strongly convex domain $\Om$ in the plane, it is known that 
    \be\label{eq: wights plane}
        \omega_\Om(z)\approx e^{-2H_\Om(z)}|z| ^{\frac{1}{2}}dd^cH_\Om(z)=\nu_\Om(z),\quad  \text{for } |z|>1.  
    \ee
However, unlike $\omega_\Om$, $\nu_\Om$ does not capture the growth rate of the Laplace transforms of Hardy space functions on other convex planar domains, such as polygons. Thus, $\omega_\Om$ has the advantage of providing a unified expression for the  Paley--Wiener spaces of Hardy spaces on convex planar domains. Although tacitly implied in \cite{Li02}, it is never shown that the weights $\omega_\Om$ and $\nu_\Om$ are similarly comparable for higher-dimensional strongly convex domains. For the sake of completeness, we provide a proof here. 

\begin{lemma}\label{le:str cvx weight comp}
Let $\Om\subset\C^n$ be a bounded strongly convex domain. For $z\in\C^n\setminus \mathbb {B}^n$,
\be
\Vert e^{\langle z,\cdot\rangle}\Vert^{2}_{\sigma_{\Om
}}\approx \frac{e^{2H_\Om(z)}}{\|z\|^{n-\frac{1}{2}}}.
\ee
\end{lemma}
\begin{proof} First, we specialize to the case of balls. Let $R>0$. We show that, for $\Vert z \Vert>1$,
\be\label{eq:ball main}
\left\lVert e^{\langle z,\cdot\rangle}\right\rVert^{2}_{\sigma_{\mathbb{B}^n(R)
}}\approx \frac{e^{2R{\|z\|}}}{\|z\|^{n-\frac{1}{2}}}.
\ee
Fix $z$ such that $\|z\|>1$.
Then,
\[
e^{-2R{\|z\|}}\Vert z\Vert^{n-\frac{1}{2}}\Vert e^{\langle z,\cdot\rangle}\Vert^{2}_{\sigma_{\mathbb B^n(R)
}}=\|z\|^{n-\frac{1}{2}}\int_{b\mathbb B^n(R)}e^{2R\|z\|\ltt\operatorname{Re}\langle\lambda,\frac{z}{R\|z\|}\rangle-1\rtt}d\sigma_{\mathbb B^n(R)
}(\lambda).
\]
After applying the change of variable $\lambda=R\zeta$, we obtain that
\[
\|z\|^{n-\frac{1}{2}}\int_{b\mathbb B^n(R)}e^{2R\|z\|\ltt\operatorname{Re}\langle\lambda,\frac{z}{R\|z\|}\rangle-1\rtt}d\sigma_{\mathbb B^n(R)
}(\lambda)\approx \|z\|^{n-\frac{1}{2}}\int_{b\mathbb B^n}e^{2R\|z\|\ltt\operatorname{Re}\langle\zeta,\frac{z}{\|z\|}\rangle-1\rtt}d\sigma_{\mathbb B^n
}(\zeta).
\]
Next, consider the unitary transformation $U_z:\Cn\rightarrow\Cn$ such that $U_z\ltt\frac{\overline{z}}{\|z\|}\rtt=\ltt1,0,\cdots,0\rtt$. After applying the change of variable $\zeta=U_z^*\ltt\eta\rtt$, combined with the fact that $\sigma_{b\mathbb B^n}$ is invariant under unitary transformations, we obtain that 
\begin{align*}
\int_{b\mathbb B^n}e^{2R\|z\|\ltt\operatorname{Re}\langle\zeta,\frac{z}{\|z\|}\rangle-1\rtt}d\sigma_{\mathbb B^n
}(\zeta)&=\int_{b\mathbb B^n}e^{-2R\|z\|\ltt1-\operatorname{Re}\left\langle\eta,\ltt 1,0,\cdots,0 \rtt\right\rangle\rtt}d\sigma_{\mathbb B^n}(\eta)\\
&=\int_{b\mathbb B^n}e^{-2R\|z\|\ltt1-\operatorname{Re}\eta_1\rtt}d\sigma_{\mathbb B^n}(\eta),
\end{align*}
where $\eta=\ltt\eta_1,\eta_2,\cdots,\eta_n\rtt$.
Using spherical co-ordinates on $b\mathbb b\mathbb B^n$, we obtain that
\[
\int_{b\mathbb B^n}e^{-2R\|z\|\ltt1-\operatorname{Re}\eta_1\rtt}d\sigma_{b\mathbb B^n}(\eta)\approx \int_0^\pi e^{-2R\|z\|\ltt1-\cos \theta\rtt}\ltt\sin\theta\rtt^ {2n-2} d\theta.
\]
When $\theta\in\ltt\frac{\pi}{2},\pi\rtt$, $1<\ltt1-\cos(\theta_1)\rtt<2$. Thus,
\be\label{eq: ball estimate away from 0}
0\leq\|z\|^{n-\frac{1}{2}}\int_{\frac{\pi}{2}}^{\pi} e^{-2R\|z\|\ltt1-\cos \theta\rtt}\ltt\sin\theta\rtt^ {2n-2}d\theta\lesssim 1.
\ee
For $\theta\in (0,\frac{\pi}{2})$,  $\sin\theta\approx \theta.$ Thus, for all $\theta\in (0,\frac{\pi}{2})$, we have that
\[
\int_0^{\frac{\pi}{2}} e^{-\alpha \theta^2}\theta^ {2n-2}d\theta\lesssim\int_0^{\frac{\pi}{2}} e^{-2R\|z\|\ltt1-\cos \theta\rtt}\ltt\sin\theta\rtt^ {2n-2}d\theta\lesssim\int_0^{\frac{\pi}{2}} e^{-\beta \theta^2}\theta^ {2n-2}d\theta,
\]
where $\alpha=C_1\|z\|$ and $\beta=C_2\|z\|$ for some constants $C_1,C_2>0$ that are independent of $z$. \newline
Fix a $k>0$, then applying the change of variable $x=\sqrt{\gamma}\theta$, we have that, for $\gamma\geq k$,
\[
\int_0^{\frac{\pi}{2}} e^{-\gamma\theta^2}\theta^ {2n-2}d\theta
=\gamma^{\frac{1}{2}-n} \int_0^{\frac{\sqrt{\gamma}\pi}{2}} e^{-x^2}  x^{2n-2}dx\approx \gamma^{\frac{1}{2}-n}.\]
Hence, 
\be\label{eq:ball est near 0}
\|z\|^{n-\frac{1}{2}}\int_0^{\frac{\pi}{2}} e^{-2R\|z\|\ltt1-\cos \theta\rtt}\ltt\sin\theta\rtt^ {2n-2}d\theta\approx 1.
\ee
Finally, combining \eqref{eq: ball estimate away from 0} and \eqref{eq:ball est near 0}, we conclude the proof in the special case of balls.

Let $\Om\subset\Cn$ be a bounded strongly convex domain. Without loss of generality, assume $0\in \Om$. We collect some facts about the support function $H_\Om$ of $\Om$; see \cite[\S 2.5]{Sc14}. Note that the definition of support function $h_\Om$ considered in \cite{Sc14} is different from $H_\Om$. However, they are related by $H_\Om(z)=h_\Om(\bar{z})$, $z\in \C^n$.   
\begin{enumerate}
\item  $H_\Om$ is $\cont^2$ smooth on $\C^n\setminus\{0\}$.
\item The map $\lambda\rightarrow 2\partial H_\Om(\overline{\lambda})$ is a $\cont^1$-difeomorphsim from $b\mathbb B^n$ onto $b\Om$.
\item Owing to the $1$-homegeneity of $H_\Om$, we may write (abusing notation):
\be\label{eq:identity of H}
H_\Om(z)=2\operatorname{Re}\langle\partial H_\Om(z),z\rangle, \quad z\in\Cn\setminus\{0\}.
\ee

\item For any real-valued function $f$ that is measurable with respect to $\sigma_\Om$,
\be\label{eq:cov boundary to sphere}
\int_{b\Om}f(\zeta)d\sigma_{\Om}(\zeta)\approx \int_{b\mathbb B^n} f(2\partial H_\Om(\overline\lambda))d\sigma_{\mathbb B^n}(\lambda).
\ee

\item The polar domain $\Om^\circ$ of $\Om$ given by $   \Om^\circ=\left\{z:H_\Om(\overline z)<1 \right\}$  is also strongly convex. 
\end{enumerate} 

By \eqref{eq:cov boundary to sphere} and  \eqref{eq:identity of H}, we have that 
\be\label{eq:apriori est}
\frac{1}{e^{2H_\Om(\overline z)}}\int_{b\Om}e^{2\operatorname{Re}\langle\zeta,\overline z\rangle}d\sigma_\Om(\zeta)\approx \int_{b\mathbb B^n} e^{4\|z\|\ltt \operatorname{Re}\left\langle\partial H_\Om(\overline\lambda)-\partial H_\Om(\overline z),\frac{\overline z}{\|z\|}\right\rangle\rtt} d\sigma_{\mathbb B^n}(\lambda).
\ee
Note that, using \eqref{eq:identity of H} twice, we obtain that for $\ltt\lambda,z\rtt\in b\mathbb B^n\times \C^n\setminus\{0\}$,
\bea
\operatorname{Re}\left\langle\partial H_\Om(\overline z)-\partial H_\Om(\overline\lambda),\frac{\overline z}{\|z\|}\right\rangle&\approx&
\operatorname{Re}\left\langle\partial H_\Om(\overline z)-\partial H_\Om(\overline\lambda),\frac{\overline z}{H_\Om(\overline z)}\right\rangle
\notag \\&=&
\frac{1}{2}-\operatorname{Re}\left\langle\partial H_\Om(\overline\lambda),\frac{\overline{\lambda}}{H_\Om(\overline\lambda)}-\frac{\overline{\lambda}}{H_\Om(\overline\lambda)}+\frac{\overline z}{H_\Om(\overline z)}\right\rangle
\notag \\&=&
\operatorname{Re}\left\langle\partial H_\Om(\overline\lambda),\frac{\overline{\lambda}}{H_\Om(\overline\lambda)}-\frac{\overline z}{H_\Om(\overline z)}\right\rangle\notag \\
&=& \operatorname{Re}\left\langle\partial H_\Om\left(\frac{\overline{\lambda}}{H_\Om(\overline\lambda)}\right),\frac{\overline{\lambda}}{H_\Om(\overline\lambda)}-\frac{\overline{z}}{H_\Om(\overline z)}\right\rangle.\label{eq:exponent}
\eea
Now, by the definition of $\Om^\circ$, it is clear that for any $w\neq0$, $\frac{w}{H_\Om(\overline w)}\in b\Om^\circ$. Considering $H_\Om(\bar{z})-1$ to be the defining function of $\Om^\circ$, we get from \cite[Lemma 5.1]{Fo86} that, for all $(\lambda,z)\in b\mathbb B^n\times \C^n\setminus\{0\}$,
\be\label{eq:exponent2}
\operatorname{Re}\left\langle\partial H_\Om\left(\frac{\overline{\lambda}}{H_\Om(\overline\lambda)}\right),\frac{\overline{\lambda}}{H_\Om(\overline\lambda)}-\frac{\overline{z}}{H_\Om(\overline z)}\right\rangle\approx \left\|\frac{z}{H_\Om(\overline z)}-\frac{\lambda}{H_\Om(\overline\lambda)}\right\|^2.
\ee
Moreover, for all $(\lambda, z)\in  b\mathbb B^n\times \C^n\setminus\{0\}$,
\be\label{eq:exponent3}
\left\|\frac{z}{H_\Om(\overline z)}-\frac{\lambda}{H_\Om(\overline{\lambda})}\right\|^2 \approx \left\|\frac{z}{\Vert z\Vert}-\lambda\right\|^2= 2\operatorname{Re}\left\langle\frac{z}{\|z\|}-\lambda, \frac{\overline{z}}{\|z\|}\right\rangle.
 \ee
Combining \eqref{eq:exponent}, \eqref{eq:exponent2}, \eqref{eq:exponent3}, and \eqref{eq:apriori est}, we note that it suffices to estimate 
\bes
\int_{b\mathbb B^n}e^{2C\|z\|\ltt\operatorname{Re}\left\langle\lambda,\frac{\bar{z}}{\|z\|}\right\rangle-1\rtt}d\sigma_{\mathbb B^n}(\lambda)=C^{1-2n}\Vert e^{\langle \overline{z},\cdot\rangle}\Vert^{2}_{\sigma_{\mathbb B^n(C)
}}e^{-2C\|z\|},
\ees
where $C>0$ is some constant. Now, applying \eqref{eq:ball main} to the right-hand side, we conclude the proof.  
\end{proof}

\begin{proposition}
\label{pr:str cvx weight} Let $\Om\subset\C^n$ be a bounded strongly convex domain. Then, the identity map is a normed space isomorphism between $
A^2(\C^n,\omega_\Om)$ and $A^2(\C^n,\nu_\Om)$.    
\end{proposition}
\begin{proof} We must show that
\[
\int_{\mathbb B^n} |F(z)|^2d\omega_\Om(z)\approx \int_{\mathbb B^n} |F(z)|^2 d\nu_\Om(z),
\]
uniformly for all $F\in\hol\ltt\C^2\rtt.$ Due to Lemma~\ref{le:str cvx weight comp}, it suffices to prove that
  \[
  \int_{\mathbb B^n} |F(z)|^2\ltt dd^cH_\Om\rtt^n(z)\approx \int_{\mathbb B^n} |F(z)|^2 \|z\|^{n-\frac{1}{2}}\ltt dd^cH_\Om\rtt^n(z),
  \]
  where $F$ is a holomorphic polynomial. This can be verified easily using spherical coordinates and the fact that $\|z\|^n\ltt dd^cH_\Om\rtt^n\approx \omega_{std}$. 
  \end{proof}

\bibliography{24128FIN}{}
\bibliographystyle{plain}

\end{document}